\documentclass[11pt,a4paper]{amsart}
\usepackage{amsmath}
\usepackage{amssymb}
\usepackage{graphicx}
\usepackage[all]{xy}
\usepackage{latexsym}
\usepackage{mathrsfs}
\usepackage{fancyhdr}
\usepackage{enumerate}
\usepackage[english]{babel}
\usepackage{lipsum}
\usepackage{etoolbox}
\usepackage{colonequals}
\usepackage{multicol}
\usepackage[T1]{fontenc}
\usepackage[colorlinks=true,linkcolor=Maroon, citecolor=LimeGreen, pagebackref]{hyperref} 
\usepackage[usenames,dvipsnames]{xcolor}
\usepackage{lipsum}
\usepackage{mwe}
\usepackage{enumitem}
\usepackage{float}
\usepackage[normalem]{ulem}
 \usepackage{tikz}
\usetikzlibrary{arrows,decorations.pathmorphing,decorations.pathreplacing,positioning,shapes.geometric,shapes.misc,decorations.markings,decorations.fractals,calc,patterns}

\tikzset{>=stealth',
     cvertex/.style={circle,draw=black,inner sep=1pt,outer sep=3pt},
     vertex/.style={circle,fill=black,inner sep=1pt,outer sep=3pt},
     star/.style={circle,fill=yellow,inner sep=0.75pt,outer sep=0.75pt},
     tvertex/.style={inner sep=1pt,font=\criptsize},
     gap/.style={inner sep=0.5pt,fill=white}}

\newcommand{\arrowrl}[3][20]
{
\hspace{-5pt}
\begin{tikzpicture}
\node (A) at (0,0) {};
\node (B) at (1,0) {};
\draw[->] ($(A)+(0,0.2)$) -- node [above] {$\scriptstyle f^*$} ($(B)+(0,0.2)$);
\draw [->] ($(B)+(0,0.2)$) -- node [below] {$\scriptstyle f_*$} ($(A)+(0,0.2)$);
\end{tikzpicture}
\hspace{-5pt}
}
\newcommand{\adj}[2][20]{\arrowrl}

\newcommand{\bC}{\mathbb{C}}

\newcommand{\bN}{\mathbb{N}}

\newcommand{\bZ}{\mathbb{Z}}

\newcommand{\cC}{\mathcal{C}}
\newcommand{\cT}{\mathcal{T}}
\newcommand{\cA}{\mathcal{A}}

\newcommand{\mcm}{\mathrm{MCM}}

\usepackage[all]{xy}
 \usepackage{tikz-cd}

\usepackage[left=2.5cm, top=3cm, right=2.5cm, bottom=3cm, bindingoffset=0.5cm]{geometry}

\newcommand{\Z}{\ensuremath{\mathbb{Z}}}
\newcommand{\C}{\ensuremath{\mathbb{C}}}

\newcommand{\CC}{\ensuremath{\mathcal{C}}} 
\newcommand{\T}{\ensuremath{\mathcal{T}}} 
\newcommand{\cF}{\mathcal{F}}

\DeclareMathOperator{\add}{add}

\DeclareMathOperator{\Coker}{Coker}

\DeclareMathOperator{\Hom}{Hom}
\DeclareMathOperator{\Ext}{Ext}

\DeclareMathOperator{\Spec}{Spec}

\newcommand{\CM}{\operatorname{{MCM}}}

\newcommand{\mc}[1]{\ensuremath{\mathcal{#1}}}   

\usepackage{collectbox}    

\makeatletter
\newcommand{\sqbox}{%
    \collectbox{%
        \@tempdima=\dimexpr\width-\totalheight\relax
        \ifdim\@tempdima<\z@
            \fbox{\hbox{\hspace{-.5\@tempdima}\BOXCONTENT\hspace{-.5\@tempdima}}}%
        \else
            \ht\collectedbox=\dimexpr\ht\collectedbox+.5\@tempdima\relax
            \dp\collectedbox=\dimexpr\dp\collectedbox+.5\@tempdima\relax
            \fbox{\BOXCONTENT}%
        \fi
    }%
}
\makeatother

\theoremstyle{theorem}
\newtheorem{theorem}{Theorem}[section]         
\newtheorem{lemma}[theorem]{Lemma}
\newtheorem{cor}[theorem]{Corollary}
   
\newtheorem{proposition}[theorem]{Proposition}

\theoremstyle{remark}
\newtheorem{remark}[theorem]{Remark}

\newtheorem{ex}[theorem]{Example}

\theoremstyle{definition}
\newtheorem{dfn}[theorem]{Definition} 

\numberwithin{equation}{section}

\title{Cluster structures for the $A_\infty$ singularity}

\author[August]{Jenny August}

\author[Cheung]{Man Wai Cheung}

\author[Faber]{Eleonore Faber}

\author[Gratz]{Sira Gratz}

\author[Schroll]{Sibylle Schroll}

\date{\today}
\begin{document}

\begin{abstract}
We study a category $\cC_2$ of $\bZ$-graded MCM modules over the $A_\infty$ curve singularity and demonstrate it has infinite type $A$ cluster combinatorics. In particular, we show that this Frobenius category (or a suitable subcategory) is stably equivalent to the infinite type $A$ cluster categories of Holm--J{\o}rgensen, Fisher and Paquette--Y{\i}ld{\i}r{\i}m. As a consequence, $\cC_2$ has cluster tilting subcategories modelled by certain triangulations of the (completed) $\infty$-gon. We use the Frobenius structure to extend this further to consider maximal almost rigid subcategories, and show that these subcategories and their mutations exhibit the combinatorics of the completed $\infty$-gon.
\end{abstract}

\maketitle

\section{Introduction}
    
    Throughout cluster theory, type $A$ serves as a prototypical example to understand new concepts. This paper formalises the connection between infinite rank type $A$ (completed) cluster combinatorics and the corresponding plane curve singularity. 

    This builds on work by \cite{JKS16} Jensen, King and Su, who give a correspondence between finite type $A$ cluster algebras and the corresponding hypersurface singularities. More precisely, a category of maximal Cohen-Macaulay modules over a curve singularity of type $A_n$ encodes the combinatorics of a cluster algebra of type $A_n$. A key aspect of this combinatorics is that it is fully described by triangulations of a regular $(n+3)$-gon. In the categorical setting, indecomposable objects correspond to arcs in the polygon, and triangulations give cluster tilting objects. Consequently, mutation is encoded by diagonal flips, with exchange sequences determined by the ambient quadrilateral.
    
    From a combinatorial perspective, it is natural to extend this to an $\infty$-gon: Take a discrete set of points on the unit circle $S^1$ with one two-sided accumulation point. We may think of the points as being indexed by the integers. The notions of arcs, triangulations and diagonal flips extend directly from the finite case, with arcs being labelled by pairs of integers. A central challenge of this infinite rank setting is that the exchange graph of triangulations under finite sequences of diagonal flips is no longer connected. Combinatorially, this issue can be fixed by instead considering transfinite mutations in the \emph{completed} $\infty$-gon, as studied by Baur and Gratz in \cite{BG}, and independently by \c{C}anak\c{c}{\i} and Felikson in \cite{CF19}. In the completed $\infty$-gon we label the accumulation point by $-\infty$, and allow arcs between $-\infty$ and any integer, called \emph{infinite arcs}. In this paper we prove that a category associated to the type $A_\infty$ curve singularity exhibits both the combinatorics  of the $\infty$-gon and of the completed $\infty$-gon.
    
    Specifically, we consider the category $\cC_2$ of $\mathbb{Z}$-graded maximal Cohen-Macaulay modules over $\mathbb{C}[x,y]/(x^2)$ with $x$ in degree $1$ and $y$ in degree $-1$. This is the Grassmannian category of infinite rank associated to an infinite version of the Grassmannian of planes, as introduced by the authors in \cite{ourfirstpaper}. We focus on this specific Grassmannian category because it is tame, allowing us to explicitly describe and classify cluster tilting and related subcategories, as well as their mutations.

    Constructing cluster categories with infinite type $A$ combinatorics is a natural problem which has been tackled from different angles throughout the literature, starting with the pioneering paper by Holm and J{\o}rgensen \cite{HJ-cat}. They prove that the finite derived category $\mathrm{D^f}(S)$ of differential graded (dg) modules over the dg-algebra $S=\mathbb{C}[y]$ (with $y$ in cohomological degree $-1$) exhibits the cluster combinatorics of the $\infty$-gon. We prove that a natural subcategory of $\cC_2$ recovers the Holm-J{\o}rgensen category $\mathrm{D^f}(S)$. More precisely, we show the following.
    
    \begin{theorem}[Corollary \ref{C: link to HJ}]
        Denote by $\cC_2^{\mathrm{f}}$ the subcategory of generically free modules in $\cC_2$ (see Definition \ref{D: generically free}). Then its stable category $\underline{\cC}_2^{\mathrm{f}}$ is equivalent to $\mathrm{D^f}(S)$.
    \end{theorem}

    In \cite{ourfirstpaper} we show that in general the category of generically free modules of a Grassmannian category of infinite rank exhibits the combinatorics of Pl\"ucker coordinates in the homogeneous coordinate ring of the corresponding Grassmannian. This subcategory is always $2$-Calabi-Yau, and as is such an ideal setting to study rigidity and cluster tilting. 
    
    The category $\cC_2$ coming from the $A_\infty$ curve singularity is a particularly well-behaved example of a Grassmannian category of infinite rank:  It has been studied as an isolated line singularity by Siersma \cite{Siersma} in the 1980s and it was shown by Buchweitz--Greuel--Schreyer \cite{BGS} that type $A_\infty$ and $D_\infty$ are the only hypersurface singularities of countable Cohen--Macaulay type. In particular, \cite{BGS} classifies the isomorphism classes of indecomposable MCM-modules via matrix factorizations in both cases. In this paper, we show that there is a full classification of the indecomposable objects in $\cC_2$ given by arcs in the completed $\infty$-gon. That same combinatorics was discovered in a different set-up by Fisher in \cite{F16}, who completed the Holm--J{\o}rgensen category under homotopy colimits to obtain a triangulated subcategory $\overline{\mathcal{D}}$ of the derived category of the dg-algebra $S$.
    
    Igusa and Todorov \cite{ITcyclic} have generalised the idea of cluster categories of infinite type $A$ in a combinatorial manner, by extending the notion of $\infty$-gon. Building on this, Paquette and Y{\i}ld{\i}r{\i}m \cite{PY21} present a combinatorial completion, yielding triangulated categories now containing indecomposable objects corresponding to arcs starting or ending in accumulation points. In particular, in the one-accumulation point case, the indecomposable objects in the Paquette-Y{\i}ld{\i}r{\i}m category $\overline{\cC}_M$ are indexed in the same way as the indecomposable objects in $\cC_2$ as well as in the Fisher category. In fact, we prove the following.
    
    \begin{theorem}[Propositions \ref{P:Fisher}, \ref{P:PY}]\label{T:Equivalence2}
        There are equivalences of triangulated categories
        \[
            \underline{\cC}_2 \cong \overline{\cC}_M \cong \overline{\mathcal{D}},
        \]
        where $\underline{\cC}_2$ denotes the stable category of the Grassmannian category $\cC_2$.
    \end{theorem}

    The categories $\overline{\cC}_M$ and $\overline{\mathcal{D}}$ were constructed explicitly with the goal to obtain a categorical analogue to the combinatorics of a completed $\infty$-gon. Our result
    proves that these combinatorics are not at all artificial: They actually occur in a very natural way in an algebro-geometric setting.
    Furthermore, we immediately get a classification of cluster tilting subcategories in $\cC_2$ and $\cC^\mathrm{f}_2$ via lifting the results \cite[Theorem 4.4]{PY21}, \cite[Theorem 5.11]{F16} and \cite[Theorem 4.4]{HJ-cat} from the triangulated category $\underline{\cC}_2$ to its Frobenius cover $\cC_2$, cf.\ Theorems \ref{t: complete ct} and \ref{t: not complete ct}. Note that these classifications can also be recovered via straightforward computations in $\cC_2$, which are detailed in Appendix \ref{appendix:A}.
    
    In both $\cC_2$ and $\cC_2^{\mathrm{f}}$, cluster tilting subcategories correspond to certain triangulations. This is a consequence of the fact that a crossing of arcs corresponds to the non-vanishing of the $\Ext^1$-group between the respective indecomposable objects. In the latter category, which is $2$-Calabi-Yau, the converse is also true. Interestingly, in the former category, there are extensions between any two distinct infinite arcs. Therefore, any triangulation containing more than one infinite arc corresponds to a category which is not rigid. However, in order to fully describe the combinatorics of the $\infty$-gon, we want to study all triangulations from a categorical perspective. This leads us to consider maximal almost rigid subcategories (see Definition \ref{dfn: mar}), inspired by maximal almost rigid objects introduced by Barnard, Gunawan, Meehan and Schiffler \cite{BGMS} .
    
    \begin{theorem}[Theorem \ref{thm: cat of triang}]
            A subcategory $\cA \subset \cC_2$ is maximal almost rigid if and only if its indecomposable objects correspond to a triangulation of the completed $\infty$-gon.
    \end{theorem}

    Note that the notion of almost rigidity does not behave well under stabilisation, and relies on an exact structure. As such, it is crucial that we work in the category $\cC_2$, and not in the triangulated categories from Theorem \ref{T:Equivalence2}.
    
    Mutation in cluster categories is designed to mirror cluster algebra combinatorics. The notion of mutation in triangulated categories with respect to rigid subcategories has been introduced by Iyama and Yoshino in \cite{IYmutation}, and mutation specifically for cluster structures was studied by Buan, Iyama, Reiten and Scott in \cite{BIRS}. We extend these concepts to mutation of maximal almost rigid subcategories in $\cC_2$, see Definition \ref{mutationdefinition}. This allows for mutation with respect to non-rigid subcategories, as well as for the mutability of an indecomposable object to vary with respect to its ambient cluster.
    
    \begin{theorem}[Theorem \ref{t: mutation}]
    Let $\cA$ be a maximal almost rigid subcategory of $\cC_2$ corresponding to a triangulation $T$ of the completed $\infty$-gon. An indecomposable object $X$ of $\cA$ is mutable if and only if the corresponding arc $\gamma$ in $T$ is mutable. Furthermore, the mutation of $\cA$ at $X$ corresponds to the mutation $\mu_\gamma(T)$ of $T$ at $\gamma$.
\end{theorem}

In particular, the combinatorics of maximal almost rigid subcategories in $\cC_2$ is precisely the combinatorics of triangulations of the completed $\infty$-gon.
As a direct consequence, we obtain connectivity of the exchange graph of maximal almost rigid subcategories under transfinite mutation.

\begin{cor} [Corollary \ref{c: connected}]
    The exchange graph of maximal almost rigid subcategories of $\cC_2$ is connected.
\end{cor}

\subsubsection*{Acknowledgements} The authors would like to thank the Isaac Newton Institute for Mathematical Sciences, Cambridge, for support (through EPSRC grant EP/R014604/1) and hospitality during the programme Cluster algebras and representation theory  where work on this paper was undertaken. This project started from the WINART2 (Women in Noncommutative Algebra and Representation Theory) workshop, and the authors would like to thank the organisers for this wonderful opportunity. They also thank the London Mathematical Society (WS-1718-03), the University of Leeds, the National Science Foundation (DMS 1900575), the Association for Women in Mathematics (DMS-1500481), and the Alfred P. Sloan foundation for supporting the workshop. S.G. acknowledges support from Villum Fonden and from the EPSRC (Grant Number EP/V038672/1). S.S. acknowledges support from  EPSRC (Grant Number EP/P016294/1). J.A. acknowledges support from the Max Planck Institute of Mathematics and a DNRF Chair from the Danish National Research Foundation (Grant Number DNRF156).
    
    \section{The $A_\infty$ curve singularity} \label{S:Grassmannian category} 
    
    Let  $R=\mathbb{C}[x,y]/(x^2)$. This is a nonreduced hypersurface ring, and $\Spec(R)$ is a plane curve singularity of type $A_\infty$. This terminology should indicate that $\Spec(R)$ is a ``limit'' of singularities of type $A_n$, that are defined as $\Spec(\C[x,y]/(x^2+y^{n+1}))$. The singular locus of $\Spec(R)$ is one-dimensional and this type of singularity has been studied by Siersma in the 1980s, see \cite{Siersma} for more info.
    
    Following \cite{ourfirstpaper}, we will consider $R$ as a graded ring with 
    $x$ in degree $1$ and $y$ in degree $-1$, and define
    $\mathcal{C}_{2} \colonequals \mcm^{\bZ} R$, which is the category of finitely generated $\bZ$-graded MCM $R$-modules. This is a natural generalisation of the Grassmannian cluster categories studied by Jensen--King--Su in the finite setting. This category is Krull--Schmidt and Frobenius, where the projective-injectives are given by all graded shifts of $\add(R)$, and thus the stable category $\underline{\cC}_2$ is a triangulated category with shift functor given by the inverse syzygy $\Omega^{-1}$.

\subsection{Classification of objects} From an algebraic point of view, $R$ is of \emph{countable $\CM$-type}, which means that it has only countably many isomorphism classes of indecomposable MCM-modules, see e.g.~the book by Leuschke and Wiegand \cite[Chapter 14]{LeuschkeWiegand} for more details. In particular, by a theorem of Buchweitz, Greuel and Schreyer \cite{BGS}, any hypersurface ring of countable $\CM$-type over $\C$ is either of type $A_\infty$ or $D_\infty$, see \cite[Theorem 14.16]{LeuschkeWiegand}. Surfaces of countable $\CM$-type have been studied by Burban and Drozd in \cite{BurbanDrozd}.

Throughout the paper, given a graded module $M$, we will use $M(j)$ to denote the graded module with $M(j)_n=M_{j+n}$. 

\begin{proposition} \label{prop: classify objects}
Let $M$ be an indecomposable graded $\CM$-module over $R$. Then $M$ is determined by a graded shift of one of the following matrix factorizations of $x^2$:
\begin{enumerate}
\item For $A=x^2$ and $B=0$ we get the graded matrix factorization of rank 1
$$ \mathbb{C}[x,y](-2) \xrightarrow{1} \mathbb{C}[x,y](-2) \xrightarrow{x^2} \mathbb{C}[x,y] \ , $$
which gives $M=\Coker(A)=R$.
\item For $A=B=x$, we get again a graded matrix factorization of rank 1
$$ \mathbb{C}[x,y](-2) \xrightarrow{x} \mathbb{C}[x,y](-1) \xrightarrow{x} \mathbb{C}[x,y] \ , $$
which gives $M=\Coker(A)=R/(x)\cong \C[y]$.
\item For each $k \in \Z_{> 0}$ and $A=B=\begin{pmatrix} x & y^k \\ 0 & -x \end{pmatrix}$ we get a graded matrix factorization of rank 2 
$$ \mathbb{C}[x,y](-2)\oplus \mathbb{C}[x,y](k-1) \xrightarrow{\phantom{aa}B\phantom{aa}} \mathbb{C}[x,y](-1) \oplus \mathbb{C}[x,y](k) \xrightarrow{\phantom{aa}A\phantom{aa}} \mathbb{C}[x,y] \oplus \mathbb{C}[x,y](k+1) \ , $$
giving $M=\Coker(A) \cong (x,y^k)$.
\end{enumerate}
\end{proposition}

\begin{proof}
The matrix factorizations were determined in \cite[Prop.~4.1]{BGS} in the local case, that is, instead of $\C[x,y]$ one considers $S$ a noetherian regular local ring. All these matrix factorizations are gradable, so they are also a complete set of reduced matrix factorizations in our case. The degrees in all cases can be calculated from the matrix presentations, using that $\deg (x^2)=2$. Moreover, in (3), the isomorphism of $\Coker(A)$ with the ideal $(x,y^k)$ can be seen by a direct calculation using that $\Coker(A) \cong \ker(A)$. 
\end{proof}
\begin{remark} \label{R : suspension=gradingshift}
Note that the modules in (2) and (3) satisfy $M(-1) \cong \Omega(M)$. In particular this shows that the grading shift and suspension functor in $\underline{\cC}_2$ will coincide.  
\end{remark}
Since we may think of the graded module $R(j)$ as the ideal $(x,y^0)(j)$, we will consider the objects of $\cC_2$ as lying in two families: those of the form $\C[y](j)$ and those of the form $(x,y^k)(j)$.

\subsection{The subcategory $\cC_2^\mathrm{f}$}

Continuing to follow \cite{ourfirstpaper}, we will be interested in a particular subcategory of $\cC_2$ for which the indecomposable objects are in bijection with the Pl\"ucker coordinates of the corresponding Grassmannian cluster algebra. To define this subcategory we need the following. 

Let $\cF$ be the graded total ring of fractions of $R$, i.e.\ 
 the ring $R$ localised at all homogeneous non-zero divisors:
 \[
 	\cF = R_y = \bC[x,y^{\pm}]/(x^2).
 \]

We consider $\cF$ as a graded ring, with the grading induced by the grading of $R$. 

\begin{dfn} \label{D: generically free}
    A module $M \in \mathrm{gr} R$ is {\em generically free of rank $n$} if $M\otimes_R \cF$ is a graded free $\cF$-module of rank $n$.
\end{dfn}

We call the subcategory of $\cC_2$ consisting of the generically free modules $\cC_2^\mathrm{f}$, and note that this is an extension closed subcategory which is stably 2-Calabi-Yau by \cite[Proposition 3.12]{ourfirstpaper}. Moreover, we know by \cite[Theorem 3.7]{ourfirstpaper} that the generically free modules in $\cC_2$ are precisely the shifted ideals $(x,y^k)(j)$ for $k \geq 0$, $j \in \mathbb{Z}$, and these correspond to Pl\"ucker coordinates in the homogeneous coordinate ring of an infinite version of the Grassmannian of planes.

In the following sections, we will see how both $\cC_2$ and $\cC_2^\mathrm{f}$ can be considered as cluster categories of type $A_\infty$.

\section{Equivalences of Categories}
    
 In this section, we observe the close connection between $\mathcal{C}_{2} $ and other versions of infinite type $A$ cluster categories. The key strategy will be to link all the categories we consider to a particular differential graded (dg) algebra.

\subsection{The dg algebra $\mathbb{C}[y]$}\label{S:dg algebra}

Consider the dg algebra $S = \mathbb{C}[y]$ with zero differential, and with $y$ in cohomological degree $-1$. We will show that all the categories we are studying are (stably) equivalent to (a subcategory of) $\mathrm{Perf}(S)$.

Since $x$ squares to $0$ in $R=\C[x,y]/(x^2)$, we can view every graded $R$-module $M$ as a dg module over $S$, with differential given by the action of $x$, and vice versa. Note that the degrees agree: The action of $x$ (respectively the differential) increases the degree by $1$, while the action by $y$ decreases the degree by $1$. This yields an equivalence of categories
\[
	\mathrm{gr} R \cong \mathrm{dg} S,
\] 
between finitely generated $\bZ$-graded $R$-modules and finitely generated dg-modules over the dg algebra $S$. A generically free indecomposable graded MCM  $R$-module isomorphic to $(x, y^i) (j)$ for $i \in \bZ_{\geq 0}, j \in \bZ$ corresponds to the isomorphism class of the dg $S$-module 
\[
	\ldots \xrightarrow{x} \langle xy^{i+2},y^{i+1} \rangle  \xrightarrow{x} \langle xy^{i+1},y^i \rangle  \xrightarrow{x} \langle xy^i \rangle \xrightarrow{x}  \ldots \xrightarrow{x} \ldots \xrightarrow{x} \langle xy^2 \rangle \xrightarrow{x} \langle xy \rangle \xrightarrow{x} \langle x \rangle \xrightarrow{x} 0 \to \ldots
\]
with cohomology concentrated  in degrees $1-j-i$ to $1-j$. Here, the angled brackets denote the linear span over $\C$. Any indecomposable graded MCM $R$-module that is not generically free is isomorphic to $\mathbb{C}[y](j)$ for some $j \in \bZ$ and corresponds to the isoclass of the dg $S$-module
\[
		\ldots \xrightarrow{0} \langle y^2 \rangle \xrightarrow{0} \langle y \rangle \xrightarrow{0} \langle 1 \rangle \to 0 \to \ldots
\]
with $\langle 1 \rangle$ in degree $-j$ (and where $x$, respectively the differential, acts trivially). The projective graded $R$-modules are $R(j)$ for $j \in \bZ$. This corresponds to the acyclic dg $S$-module
\[
	\ldots \xrightarrow{x} \langle xy^3, y^2 \rangle \xrightarrow{x} \langle xy^2, y \rangle \xrightarrow{x} \langle 1, xy \rangle \xrightarrow{x} \langle x \rangle \xrightarrow{x} 0 \to \ldots
\]
with $\langle x \rangle$ in degree $-j+1$. 

Recall that a graded algebra $R$ is called {\em intrinsically formal} if whenever $A$ is a dg algebra with $H^*(A) \cong R$, then $R$ is quasi-isomorphic to $A$ as a dg algebra.

\begin{lemma}\label{L:intrinsically formal}
	The algebra $S = \bC[y]$ with $y$ in degree $-1$ is intrinsically formal.
\end{lemma}

\begin{proof}
	Consider $S$ as a module over the enveloping algebra $S^e \cong \bC[y] \otimes \bC[y] \cong \bC[u,v]$.
	It has a projective resolution
	\[
		\xymatrix{0 \to S^e \ar[rr]^{1 \mapsto u-v} && S^e \ar[r]^{u,v \mapsto y} & S,}
	\]
	and thus the projective dimension of $S$ over $S^e$ is $1$. It follows by \cite{Kadeishvili}, see also \cite[Proposition 4.13]{hanihara2021morita}
	that $S$ is intrinsically formal.
\end{proof}

We will repeatedly use the following strategy to show the desired equivalences of categories.

\begin{proposition}\label{P:generator}
	Let $\mathcal{C}$ be an algebraic triangulated category with a generator $M$ with graded endomorphism ring $\mathrm{Ext}^*(M,M)$ isomorphic to $\mathbb{C}[y]$ as a graded algebra with $y$ in degree $-1$.
	Then we have an equivalence of categories
	\[
		\mathcal{C} \cong \mathrm{Perf}(S).
	\]
\end{proposition}

\begin{proof}
	By assumption, the graded endomorphism ring $\Ext^*(M,M)$ of $M$ is isomorphic as a graded algebra to the polynomial algebra $\mathbb{C}[y]$ with $y$ in degree $-1$. Now, by Lemma \ref{L:intrinsically formal}, the algebra $S$ is intrinsically formal, so that
	\[
		S \cong \Ext^*(M,M) \cong H^*(\mathrm{RHom}(M,M))
	\]
	implies that $\mathrm{RHom}(M,M)$ is quasi-isomorphic to $S$ and thus
	\[
		 \mathrm{Perf}(S) \cong \mathrm{Perf}(\mathrm{RHom}(M,M)) \cong \mathcal{C},
	\]
	where the last equivalence follows from \cite{K94}. 
\end{proof}

\subsection{$\mathcal{C}_{2}$ and the Holm--J\o rgensen Cluster Category of infinite type $A$}\label{S:HJ}

 In \cite{HJ-cat}, Holm and J\o rgensen describe the cluster structure of the finite  derived category $\mathrm{D^f}(S)$ of dg-modules over $S$. We first observe the following equivalence.
 
 \begin{proposition} \label{P: our category}
 	We have an equivalence of categories 
 	\[
	\underline{\cC}_2 \cong \mathrm{Perf}(S),
	\]
	where $\underline{\cC}_2$ denotes the stable category of $\cC_2$.
 \end{proposition}
 
 \begin{proof}
 	Note that $M = \bC[y]$ generates $\underline{\cC}_2$: For all $i \in \bZ$, the module $M(i) = \Sigma^i M$ clearly is in its thick closure. Moreover, for each $k \geq 0$ we have a short exact sequence
 	\[
 	0 \to \bC[y](-1) \to (x,y^k) \to \bC[y](k-1) \to 0
 	\]
 	in $\cC_2$, and so $ (x,y^k)(j) = \Sigma^j(x,y^k)$ is also in the thick closure of $M$. We calculate its graded endomorphism ring. 
 	The generator $M$ has a complete projective resolution
 	\[
 		\xymatrix{\ldots \ar[r] & R(-2)  \ar[r]^x & R(-1)  \ar[r]^x & R \ar[rd]  \ar[rr]^x & & R(1)  \ar[r]^x & R(2)  \ar[r]^x & \ldots \\
 		& & & & M \ar[ru] & & }
 	\]
 	Applying $\Hom(-,M)$ yields the sequence
 	\[
 		\xymatrix{\ldots \ar[r] & M_{-2}  \ar[r] & M_{-1}  \ar[r] & M_0 \ar[r]^0 & 0  \ar[r]^0 & \ldots,}
 	\]
 	where $M_i$ denotes the degree $i$ component of $M$, which vanishes for $i \geq 1$. We obtain that for $i \geq 0$,
 	\[
 		\Ext^{-i}(M,M) \cong \underline{\Hom}(M,\Omega^i M) = \underline{\Hom}(M,M(-i)) \cong \bC,
 	\]
 	and $f \in \Ext^{-i}(M,M)$ is given by multiplication with $\lambda y^i$ for some scalar $\lambda \in \bC$. We obtain that
 	\[
 		\Ext^*(M,M) \cong \mathbb{C}[y],
 	\]
 with $y$ in degree $-1$, and the statement follows from Proposition \ref{P:generator}.
 \end{proof}

Recall that $\cC_2^\mathrm{f}$ is the full subcategory of $\cC_2$ consisting of generically free modules. The indecomposable objects in $\cC_2^\mathrm{f}$ are the modules $(x,y^k)(j)$ which precisely correspond to the dg $S$-modules with finite dimensional cohomology over $\bC$. The following therefore follows immediately from Proposition \ref{P: our category}.
\begin{cor} \label{C: link to HJ}
There is an equivalence of categories
\[
	\underline{\cC}_2^\mathrm{f} \cong \mathrm{D^f}(S),
\]
where the category $\mathrm{D^f}(S)$ denotes the derived category of dg-modules over $S$ with finite dimensional cohomology over $\bC$. 
\end{cor}
This is the category studied by Holm and J\o rgensen in \cite{HJ-cat}. They show that $\mathrm{D^f}(S)$ exhibits the combinatorics of a cluster category of infinite type $A$: Indecomposable objects correspond to arcs in an $\infty$-gon, with cluster tilting objects corresponding to suitably nice triangulations thereof. Through the equivalence of Corollary \ref{C: link to HJ}, these descriptions will also extend to $\underline{\cC}_2^\mathrm{f}$. We will return to this in the next section.

\subsection{$\mathcal{C}_{2}$ and completion under homotopy colimits}\label{S:Fisher}

In \cite{F16}, Fisher completed the category $\mathrm{D^f}(S)$ under certain homotopy colimits, arriving at a triangulated category $\overline{\mathcal{D}} \subset \mathrm{D}(S)$ with indecomposable objects corresponding to arcs in the completed $\infty$-gon.

\begin{proposition}\label{P:Fisher}
	There is an equivalence of categories
	\[
		\overline{\mathcal{D}} \cong \underline{\mathcal{C}}_{2}.
	\]
\end{proposition}

\begin{proof}
	
	By Propositions \ref{P:generator} and \ref{P: our category}, it is enough to show that $\overline{\mathcal{D}}$ has a generator $M$ such that $\Ext^*_{\overline{\mathcal{D}}}(M,M)$ is isomorphic to $\C[y]$ as a graded algebra. By \cite[Definition 1.5]{F16} the only indecomposable objects in $\overline{\mathcal{D}}$, up to isomorphism, are 
	\begin{align*}
	    \{X_i(j) \mid i,j \in \Z, i \geq 0\} \cup \{E_n \mid n \in \Z\}
	\end{align*}
	where $X_i= \C[y]/(y^{i+1})$ and $E_n$ is the homotopy colimit of the direct system
	\begin{align}
	    X_0(n) \xrightarrow{y} X_1(n-1) \xrightarrow{y} X_2(n-2) \xrightarrow{y} X_3(n-3) \to \dots \label{eq: direct system}
	\end{align}
	in the derived category $D(S)$ of right dg-modules over $S=\C[y]$. We claim that $E_0$ (or in fact any $E_n$) is the required generator.

	To start, note that \eqref{eq: direct system} is also a direct system in the abelian category of right dg-modules over $S$, where the colimit can easily be calculated as the dg-module $\C[y^{-1}](n)$. In particular, there is a short exact sequence
	\begin{align*}
	    0 \to \prod_{\mathbb{N}} X_i(n-i) \xrightarrow{1-y} \prod_{\mathbb{N}} X_i(n-i) \to \C[y^{-1}](n) \to 0
	\end{align*}
	in the category of right dg modules over $S$. By definition of homotopy colimits, the induced triangle in the derived category shows that $E_n \cong \C[y^{-1}](n)$.
	
	Now, \cite[Theorem 2.8]{F16}, shows that 
	\begin{align*}
	    \Ext^i_{\overline{\mathcal{D}}}(E_0,E_0) \cong \left\{\begin{array}{lr}
        \C & \text{if } i\leq 0\\
        0 & \text{if } i > 0
        \end{array}\right.
	\end{align*}
	and, knowing $E_0 \cong \C[y^{-1}]$, one can check the map in degree -1 is the map of complexes,
     	\[
 		\xymatrix{\ldots \ar[r] & 0  \ar[r] \ar[d] & \C  \ar[r] \ar[d] & \C \ar[r] \ar[d]^1 & \C \ar[r] \ar[d]^1 & \ldots \\
 		\ldots \ar[r] & 0  \ar[r] & 0  \ar[r] & \C \ar[r] & \C \ar[r] & \ldots}
 	    \]	
 	where all differentials are zero. Since this map is clearly not nilpotent in $\Ext^*_{\overline{\mathcal{D}}}(E_0,E_0)$, the isomorphism of graded algebras, $ \Ext^*_{\overline{\mathcal{D}}}(E_0,E_0) \cong \C[y]$, follows.
	
	As $E_n=E_0(n)$, it is clear that $E_0$ generates all the $E_n$. To show that it also generates the $X_i$ note that there is a short exact sequence
	\begin{align*}
	    0 \to X_i \xrightarrow{y^{-i}} \C[y^{-1}](i) \to \C[y^{-1}](-1) \to 0
	\end{align*}
	in the category of right dg-modules over $S$. This induces a triangle in the derived category, which shows that $E_0$ generates $X_i$, and hence all of $\overline{\mathcal{D}}$ as required.
\end{proof}

\subsection{$\mathcal{C}_{2}$ and a combinatorial completion}\label{S:PY}
In \cite{PY21}, Paquette and Y{\i}ld{\i}r{\i}m present a completion of discrete cluster categories of type $A$. These are, like Igusa and Todorov's discrete cluster categories of type $A$, associated to a generalised $\infty$-gon, that is, a disc with discrete marked points on its boundary, satisfying some mild convergence condition. Indecomposable objects correspond to arcs in the generalised $\infty$-gon, that is, two-element subsets consisting of non-neighbouring marked points, and morphisms  can be read off by the respective positioning of the arcs. Unlike Igusa and  Todorov, Paquette and Y{\i}ld{\i}r{\i}m allow the accumulation points to be marked points themselves.

We show that in the ``one-accumulation point'' case, that is, when the marked points on the boundary include an unique two-sided accumulation point, Paquette and Y{\i}ld{\i}r{\i}m's construction coincides with the stable category $\underline{\mathcal{C}}_{2}$.

Let $M \subseteq S^1$ be a set of marked points on the circle with precisely one two-sided accumulation point. Denote by $\overline{\mathcal{C}}_M$ the completed cluster category in the sense of Paquette--Y{\i}ld{\i}r{\i}m \cite{PY21} (denoted $\overline{\mathcal{C}}_{(S,M)}$ there).

\begin{proposition}\label{P:PY}
	There is an equivalence of categories
	\[
		\overline{\mathcal{C}}_M \cong \underline{\mathcal{C}}_{2}.
	\]
\end{proposition}

\begin{proof}
	Let $N$ be an indecomposable object in $\overline{\mathcal{C}}_M$ corresponding to an arc $\ell_0$ connecting to the accumulation point of $M$. Denote by $\Sigma$ the suspension in $\overline{\mathcal{C}}_M$. Then for all $i \in \bZ$, the object $\Sigma^i N$ also corresponds to an arc $\ell_i$ connecting to the accumulation point of $M$, and we can go from $\ell_j$ to $\ell_i$ by rotating $\ell_j$ about the common endpoint following the orientation of the unit disc if and only if $j \geq i$. By \cite[Proposition 3.4]{PY21}, we have 
	\[
		\Ext^i(N,N) \cong \Hom(N,\Sigma^i N) \cong \begin{cases}
			 \mathbb{C}, \text{if $i \leq 0$}\\
			 0, \text{else}.
		 \end{cases}		
	\] 
By the construction of $\overline{\mathcal{C}}_M$ in \cite{PY21} and \cite[Lemma 2.4.2]{ITcyclic}, we see that for $i > 0$, a morphism $f \in \Ext^{-i}(N,N)$ factors through the $i$-fold product $\Ext^{-1}(N,N) \times \ldots \times \Ext^{-1}(N,N)$ and so the graded endomorphism ring of $N$ is isomorphic to $\bC[y]$ with $y$ in degree $-1$. The statement now follows from Figure 2 in \cite{PY21} and accompanying comments  which show that $N$ is a generator, and Proposition \ref{P:generator}. 
\end{proof}

\section{The combinatorial model and cluster tilting} \label{S:combinatorial model}
From the set-up in \cite{HJ-cat}, \cite{F16} and \cite{PY21}, we know that there is a combinatorial model for the categories described in Sections \ref{S:HJ} - \ref{S:PY} via arcs in the (completed) $\infty$-gon. In this section, we will extend this model to the Grassmannian category $\cC_2$ from Section \ref{S:Grassmannian category}.

\subsection{The completed $\infty$-gon}

An {\em $\infty$-gon} is a disc with a discrete set of marked points on the boundary admitting a unique two-sided accumulation point. We obtain a {\em completed $\infty$-gon} by adding the unique accumulation point as a marked point. In practice, we label the marked points by $\bZ$ (increasing clockwise around the disc), and call the accumulation point $-\infty$.

An {\em arc of the completed $\infty$-gon} is then a pair $(a,b)$ in $\bZ \cup \{-\infty\}$, such that $a < b$. We call an arc $(a,b)$ {\em finite} if $a,b \in \bZ$, and {\em infinite} if $a = -\infty$.
These can be illustrated as in the following pictures:

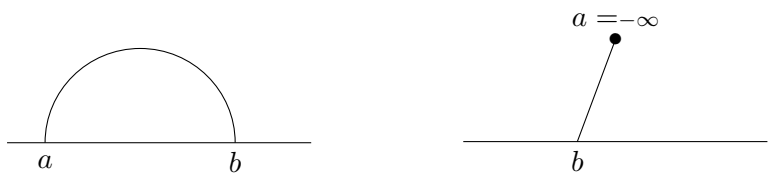
\begin{center}
\begin{tikzpicture}
\node at (-3,-.3) {
\begin{tikzpicture}[scale=0.5]
\draw (-4,0) -- (4,0);
\draw (2,0) arc (0:180:2.5);
\node at (-3,-.5) {$a$};
\node at (2,-.5) {$b$};
\end{tikzpicture}};
\node at (3,0) {
\begin{tikzpicture}[scale=0.5]
\draw (-4,0) -- (4,0);
\draw (0,2.7) -- (-1,0);
\node at (0,3.2) {$a=\scriptstyle -\infty$};
\node at (-1,-.5) {$b$};
\node at (0,2.7) {$\bullet$};
\end{tikzpicture}};
\end{tikzpicture}
\end{center}

Here, the horizontal line represents $\bZ$ and the point $-\infty$ sits separately above. When we talk about the $\infty$-gon we only consider the finite arcs, whereas the completed $\infty$-gon allows both the finite and infinite arcs.

In either case, two arcs $(a,b)$ and $(c,d)$ {\em cross} if $a<c<b<d$ or $c<a<d<b$. This notion gives rise to the following idea: A {\em triangulation of the (completed) $\infty$-gon} is a maximal set of non-crossing arcs of the (completed) $\infty$-gon. 

A triangulation $T$ is called {\em locally finite} if for all $a \in \bZ$ there are only finitely many arcs in $T$ with endpoint $a$. A set of arcs $\{(a,b_i) \mid i \in \bN\}$ is called a {\em right fountain at $a$}, if $\{b_i\}$ is a strictly increasing sequence. Similarly, a set of arcs $\{(b_i,a) \mid i \in \bN\}$ is called a {\em left fountain at $a$}, if $\{b_i\}$ is a strictly decreasing sequence. A {\em fountain at $a$} is the union of a left fountain at $a$ and a right fountain at $a$. 

For the $\infty$-gon, every triangulation is either locally finite, or contains both a left and right fountain \cite[Lemma 3.3]{HJ-cat}. For the completed $\infty$-gon, each triangulation contains precisely one of the following five configurations, where we note that each schematic triangle in the picture may contain a triangulation by finitely many arcs: 

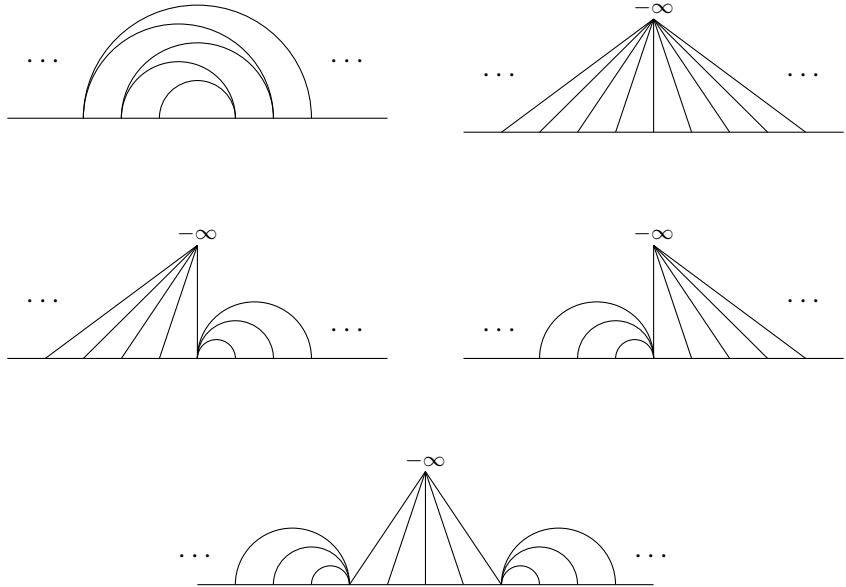
\begin{center}
\begin{tikzpicture}
\node at (-3,0) {
\begin{tikzpicture}[scale=0.5]
\draw (-5,0) -- (5,0);
\draw (1,0) arc (0:180:1);
\draw (1,0) arc (0:180:1.5);
\draw (2,0) arc (0:180:2);
\draw (2,0) arc (0:180:2.5);
\draw (3,0) arc (0:180:3);
\node at (4,1.5) {$\cdots$};
\node at (-4,1.5) {$\cdots$};
\end{tikzpicture}};
\node at (0,-6) {
\begin{tikzpicture}[scale=0.5]
\draw (-6,0) -- (6,0);
\draw (0,3) -- (0,0);
\draw (0,3) -- (1,0);
\draw (0,3) -- (-1,0);
\draw (0,3) -- (2,0);
\draw (0,3) -- (-2,0);
\draw (-2,0) arc (0:180:1);
\draw (-2,0) arc (0:180:1.5);
\draw (-2,0) arc (0:180:0.5);
\draw (4,0) arc (0:180:1);
\draw (5,0) arc (0:180:1.5);
\draw (3,0) arc (0:180:0.5);
\node at (6,0.75) {$\cdots$};
\node at (-6,0.75) {$\cdots$};
\node at (0,3.3) {$\scriptstyle -\infty$};
\end{tikzpicture}};
\node at (-3,-3) {
\begin{tikzpicture}[scale=0.5]
\draw (-5,0) -- (5,0);
\draw (0,3) -- (0,0);
\draw (0,3) -- (-1,0);
\draw (0,3) -- (-2,0);
\draw (0,3) -- (-3,0);
\draw (0,3) -- (-4,0);
\draw (2,0) arc (0:180:1);
\draw (3,0) arc (0:180:1.5);
\draw (1,0) arc (0:180:0.5);
\node at (4,0.75) {$\cdots$};
\node at (-4,1.5) {$\cdots$};
\node at (0,3.3) {$\scriptstyle -\infty$};
\end{tikzpicture}};
\node at (3,-3) {
\begin{tikzpicture}[scale=0.5]
\begin{scope}[yscale=1,xscale=-1]
\draw (-5,0) -- (5,0);
\draw (0,3) -- (0,0);
\draw (0,3) -- (-1,0);
\draw (0,3) -- (-2,0);
\draw (0,3) -- (-3,0);
\draw (0,3) -- (-4,0);
\draw (2,0) arc (0:180:1);
\draw (3,0) arc (0:180:1.5);
\draw (1,0) arc (0:180:0.5);
\node at (4,0.75) {$\cdots$};
\node at (-4,1.5) {$\cdots$};
\node at (0,3.3) {$\scriptstyle -\infty$};
\end{scope}
\end{tikzpicture}};
\node at (3,0) {
\begin{tikzpicture}[scale=0.5]
\draw (-5,0) -- (5,0);
\draw (0,3) -- (0,0);
\draw (0,3) -- (-1,0);
\draw (0,3) -- (1,0);
\draw (0,3) -- (-2,0);
\draw (0,3) -- (2,0);
\draw (0,3) -- (-3,0);
\draw (0,3) -- (3,0);
\draw (0,3) -- (-4,0);
\draw (0,3) -- (4,0);
\node at (4,1.5) {$\cdots$};
\node at (-4,1.5) {$\cdots$};
\node at (0,3.3) {$\scriptstyle -\infty$};
\end{tikzpicture}};
\end{tikzpicture}
\end{center}

This classification follows from the following observations:
\begin{itemize}
\item \cite[Lemma 1.10]{BG} If $T$ contains a left (or right) fountain at $n \in \bZ$, then $(-\infty, n) \in T$ - such an infinite arc is called a \emph{wrapping arc}.
\item \cite[Lemma 1.11]{BG} If $(-\infty, n) \in T$, then either $T$ has a left fountain at $n \in \bZ$, or there exists $m < n$ such that $(-\infty, m) \in T$.
Similarly for right fountains.
\end{itemize}
See \cite[Theorem~1.12]{BG} for a precise description of the triangulations.

\subsection{A model for $\cC_2$}\label{S:model}

Using the classification of indecomposable objects in the category $\cC_2$ from Proposition \ref{prop: classify objects}, we see they are in a natural one-to-one correspondence with the arcs of the completed $\infty$-gon in the following way:

For all $k \in \bZ_{\geq 0}$ and all $j \in \bZ$ we associate to the graded module $(x,y^k)(j)$ the finite arc $(-j-k, 1-j)$ and to $\bC[y](j)$ the infinite arc $(-\infty,-j)$. From now on, we freely use this identification, and will refer to indecomposable objects in $\cC_2$ as arcs when convenient.

Note that the boundary arcs (those of the form $(a, a+1)$) precisely correspond to the modules $R(j)$ which are the projective-injective objects in $\cC_2$. Moreover, the internal arcs correspond to indecomposable objects of $\underline{\cC}_2$ and this provides an explicit equivalence between the stable category $\underline{\cC}_2$ and the category $\overline{\cC}_M$ constructed in \cite{PY21}. As a consequence, we get the following description of Ext groups.

\begin{proposition} \label{prop:exts}
    Let $(a,b)$ and $(c,d)$ be indecomposable objects in $\CM^\Z R$. Then
    \[
        \Ext^1((a,b),(c,d)) \cong \begin{cases}
                                        \C & \text{if $(a,b)$ and $(c,d)$ cross}\\
                                        \C & \text{if $a = c = -\infty$ and $b < d$}\\
                                        0 & \text{else}.
                                    \end{cases}
    \]
\end{proposition}

\begin{proof}
    This is immediate from the equivalence between $\underline{\cC}_2$ and $\overline{\cC}_M$ and \cite[Proposition 3.14]{PY21}. Note that the computation for finite arcs also follows from \cite[Theorem C]{ourfirstpaper}, and direct calculations can be done using the matrix factorisations in Proposition \ref{prop: classify objects}.
\end{proof}

Some cases with non-zero $\Ext^1((a,b),(c,d))$ are schematically illustrated as follows: 
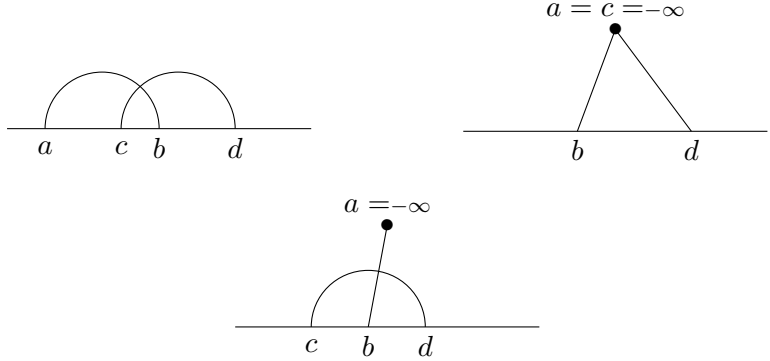
\begin{center}
\begin{tikzpicture}
\node at (-3,-.5) {
\begin{tikzpicture}[scale=0.5]
\draw (-4,0) -- (4,0);
\draw (0,0) arc (0:180:1.5);
\draw (2,0) arc (0:180:1.5);
\node at (-3,-.5) {$a$};
\node at (0,-.5) {$b$};
\node at (-1,-.5) {$c$};
\node at (2,-.5) {$d$};
\end{tikzpicture}};
\node at (3,0) {
\begin{tikzpicture}[scale=0.5]
\draw (-4,0) -- (4,0);
\draw (0,2.7) -- (-1,0);
\draw (0,2.7) -- (2,0);
\node at (0,3.2) {$a=c=\scriptstyle -\infty$};
\node at (-1,-.5) {$b$};
\node at (2,-.5) {$d$};
\node at (0, 2.7) {$\bullet$};
\end{tikzpicture}};
\end{tikzpicture}
\end{center}

\begin{center}
\begin{tikzpicture}[scale=0.5]
\draw (-4,0) -- (4,0);
\draw (1,0) arc (0:180:1.5);
\node at (-2,-.5) {$c$};
\node at (1,-.5) {$d$};
\draw (0,2.7) -- (-.5,0);
\node at (0,3.2) {$a=\scriptstyle -\infty$};
\node at (0, 2.7) {$\bullet$};
\node at (-0.5,-.5) {$b$};
\end{tikzpicture}
\end{center}

Some cases where $\Ext^1((a,b),(c,d))=0$ are illustrated as follows: 
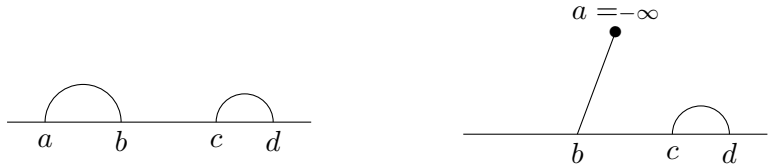
\begin{center}
\begin{tikzpicture}
\node at (-3,-.5) {
\begin{tikzpicture}[scale=0.5]
\draw (-4,0) -- (4,0);
\draw (-1,0) arc (0:180:1);
\draw (3,0) arc (0:180:.75);
\node at (-3,-.5) {$a$};
\node at (-1,-.5) {$b$};
\node at (1.5,-.5) {$c$};
\node at (3,-.5) {$d$};
\end{tikzpicture}};
\node at (3,0) {
\begin{tikzpicture}[scale=0.5]
\draw (-4,0) -- (4,0);
\draw (0,2.7) -- (-1,0);
\draw (3,0) arc (0:180:.75);
\node at (1.5,-.5) {$c$};
\node at (3,-.5) {$d$};
\node at (0,3.2) {$a=\scriptstyle -\infty$};
\node at (0, 2.7) {$\bullet$};
\node at (-1,-.5) {$b$};
\end{tikzpicture}};
\end{tikzpicture}
\end{center}

Notice that the $\Ext^1$-groups for the finite arcs are symmetric in the two arguments, corresponding to the subcategory $\cC_2^\mathrm{f}$ being stably 2-Calabi-Yau.

\begin{remark}
    Although we have used the equivalences from Sections \ref{S:HJ} - \ref{S:PY} to endow $\cC_2$ with a combinatorial model, it is worth noting that this is not necessary. Indeed, in the course of this project, we first showed $\cC_2$ had a combinatorial model by computing the $\Ext^1$-groups by hand and once we had established the model we saw the possibility of the equivalences.
\end{remark}

We may also read the Hom spaces from the combinatorial model. This is possible given the explicit description of all the indecomposable objects.

\begin{proposition}
Let $(a,b)$ and $(c,d)$ be indecomposable objects in $\CM^\Z(R)$. Then
\[
    \Hom_R((a,b),(c,d)) \cong \begin{cases}
                            \C^2 & \text{if $-\infty < a \leq c$ and $b \leq d$}\\
                            0 & \text{if $-\infty = a \leq c$ and $d<b$}\\
                            0 & \text{if $d<a$}\\
                            \C & \text{else.}
                                            \end{cases}
\]
\end{proposition}

\begin{proof}
    See Appendix \ref{S:Hom-calculations}.
\end{proof}

With the knowledge of the homomorphisms contained in Appendix \ref{appendix:A}, it is then also possible to determine the short exact sequences representing the basis elements of the $\Ext^1$-groups. In fact, these representatives can be easily read off the model, as the following three lemmas show. Note that these can be thought of as a lift of the results in \cite[Figures 1 and 2]{PY21} to the Frobenius setting.

\begin{lemma} \label{lem:ses1}
Consider two crossing finite arcs $(a,b)$ and $(c,d)$ as in the following picture:
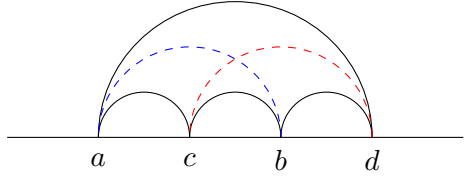
\begin{center}
    \begin{tikzpicture}[scale=0.6]
    \draw (-5,0) -- (5,0);
    \draw (1,0) arc (0:180:1);
    \draw (3,0) arc (0:180:1);
    \draw (-1,0) arc (0:180:1);
    \draw (3,0) arc (0:180:3);
    \draw[dashed,red] (3,0) arc (0:180:2);
    \draw[dashed,blue] (1,0) arc (0:180:2);
    \node at (-3,-0.5) {$a$};
    \node at (-1,-0.5) {$c$};
    \node at (1,-0.5) {$b$};
    \node at (3,-0.5) {$d$};
\end{tikzpicture}
\end{center}
Then we have non-split short exact sequences between these indecomposables given by
\begin{align*}
    0 \to (a,b) \xrightarrow{f} (c,b) \oplus (a,d) \xrightarrow{g} (c,d) \to 0\\
    0 \to (c,d) \xrightarrow{f'} (a,c) \oplus (b,d) \xrightarrow{g'} (a,b) \to 0.
\end{align*}
\end{lemma}
\begin{proof}
By direct calculation.
\end{proof}

\begin{lemma} \label{lem:ses2}
Consider a crossing between an infinite arc $(-\infty,b)$ and a finite arc $(a,c)$ as in the following picture:
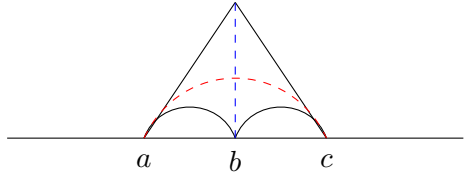
\begin{center}
\begin{tikzpicture}[scale=0.6]
\draw (-5,0) -- (5,0);
\draw (-2,0) -- (0,3);
\draw (2,0) -- (0,3);
\draw (0,0) to[looseness=1.25,out=70, in=110] (2,0);
\draw (-2,0) to[looseness=1.25,out=70, in=110] (0,0);
\draw[dashed,red] (-2,0) to[looseness=1.25,out=65, in=115] (2,0);
\draw[dashed,blue] (0,0) to (0,3);
\node at (-2,-0.5) {$a$};
\node at (0,-0.5) {$b$};
\node at (2,-0.5) {$c$};
\end{tikzpicture}    
\end{center}
Then we have non-split short exact sequences between the indecomposables given by
\begin{align*}
    0 \to (-\infty,b) \xrightarrow{f} (a,b) \oplus (-\infty,c) \xrightarrow{g} (a,c) \to 0\\
    0 \to (a,c) \xrightarrow{f'} (b,c) \oplus (-\infty,a) \xrightarrow{g'} (-\infty,b) \to 0.
\end{align*}
\end{lemma}

\begin{proof}
    By direct calculation.
\end{proof}

\begin{lemma} \label{lem:ses3}
Consider two infinite arcs $(-\infty,a)$ and $(-\infty,b)$ as in the following picture:
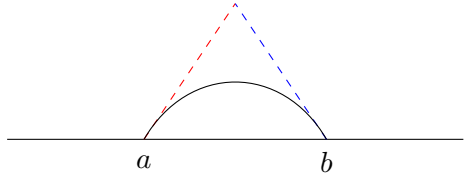
\begin{center}
\begin{tikzpicture}[scale=0.6]
\draw (-5,0) -- (5,0);
\draw [dashed,red](-2,0) -- (0,3);
\draw [dashed,blue](2,0) -- (0,3);
\draw (-2,0) to[looseness=1.25,out=60, in=120] (2,0);
\node at (-2,-0.5) {$a$};
\node at (2,-0.5) {$b$};
\end{tikzpicture}    
\end{center}
Then we have a non-split short exact sequence between the indecomposables given by
\begin{align*}
    0 \to (-\infty,b) \xrightarrow{f} (a,b) \xrightarrow{g} (-\infty,a) \to 0.
\end{align*}
\end{lemma}

\begin{proof}
    By direct calculation.
\end{proof}

\begin{remark}
By Proposition \ref{prop:exts}, the short exact sequences appearing in Lemmas \ref{lem:ses1}, \ref{lem:ses2} and \ref{lem:ses3} are, up to scalars, the only short exact sequences in $\cC_2$ with indecomposable end terms.
\end{remark}

Further, the graded shift on $\cC_2$ (and hence also the suspension on $\underline{\cC}_2$ by Remark \ref{R : suspension=gradingshift}) is easy to see in the combinatorial model:
\begin{enumerate}
    \item For a module $\C[y](j)$, the graded shift is $\C[y](j+1)$, so the shift rotates the arc $(-\infty, -j)$ to $(-\infty, -j-1)$ i.e.\ the arc is rotated one space anti-clockwise, with the point at $-\infty$ as a pivot.
    \item For a module $(x,y^i)(j)$, the graded shift is $(x,y^i)(j+1)$ so the shift rotates the arc $(-j-i,1-j)$ to $(-j-i-1,-j)$ i.e.\ the endpoints of the arc are each moved one space anti-clockwise.
\end{enumerate}

\begin{remark}\label{R:stable Hom} As a consequence, it is also easy to deduce the stable Hom spaces from the combinatorial model using Proposition \ref{prop:exts} and $\underline{\Hom}(X,Y) \cong \Ext^1(X,Y(-1))$.
\end{remark}

\subsection{Cluster tilting subcategories}

Each of the papers \cite{HJ-cat, F16, PY21} classified the cluster tilting subcategories of the relevant categories, giving these categories a cluster structure. In this section, we use the equivalences of Sections \ref{S:HJ} - \ref{S:PY} to consider $\cC_2$ and its subcategory $\cC_2^\mathrm{f}$. In all cases, however, we could have used the explicit nature of $\cC_2$ and the combinatorial model to compute the results directly.
\begin{dfn} \label{d:triangulationFrob}
Let $\cC$ be either a triangulated or Frobenius category. A full subcategory $\T$ of $\CC$ is called:
\begin{enumerate}
    \item \emph{rigid} if $\Ext^1_\CC(\T,\T)=0$;
    \item \emph{maximal rigid} if it is rigid and maximal with respect to this property i.e.\ if 
    \[
    \Ext^1_\CC(M,M)=0 \quad \text{and} \quad \Ext^1_\CC(T,M)=0=\Ext^1_\CC(M,T) 
    \]
    for all $T \in \T$, then $M \in \T$;
    \item \emph{cluster tilting} if it is functorially finite and 
    \[
    \{M \in \CC \mid \Ext^1_\CC(\T,M)=0\} = \T= \{M \in \CC \mid \Ext^1_\CC(M, \T)=0\}.
    \]
\end{enumerate}
\end{dfn}

Proposition \ref{prop:exts} makes it easy to determine the rigid subcategories in terms of certain sets of non-crossing arcs, and in each case, the cluster tilting subcategories are classified by certain \emph{triangulations}.

 For the subcategory $\cC_2^\mathrm{f}$, we have the following.
 \begin{theorem} \label{t: not complete ct}
     A subcategory of $\cC_2^\mathrm{f}$ is:
     \begin{enumerate}
         \item rigid if and only if its indecomposable objects are given by a set of non-crossing arcs in the $\infty$-gon;
         \item maximal rigid if and only if its indecomposable objects are given by a triangulation of the $\infty$-gon;
     \item a cluster tilting subcategory if and only if its indecomposable objects are given by a triangulation of the $\infty$-gon which is either locally finite, or contains a fountain at some $a \in \bZ$.
     \end{enumerate}
\end{theorem}
\begin{proof}
Note that every cluster tilting subcategory in the stable category lifts to one in the Frobenius cover when we add all projective-injective objects. The result then follows directly from \cite[Theorems A and B]{HJ-cat} and the equivalence in Corollary \ref{C: link to HJ}, and the fact that projective-injective objects correspond to boundary arcs.
\end{proof}
Note that the only other possible triangulations of the $\infty$-gon, those containing a {\em split fountain} (i.e.\ a left fountain at $a$ and a right fountain at $b$ with $a<b$), fail to be functorially finite. When we consider the whole category $\cC_2$, there are further obstructions.
\begin{theorem} \label{t: complete ct}
     A subcategory $\T$ of $\cC_2$ is a cluster tilting subcategory if and only if its indecomposable objects are given by a triangulation of the $\infty$-gon containing a fountain at some $a \in \bZ$.
\end{theorem}

\begin{proof}
    This follows immediately from \cite[Theorem 4.4]{PY21}, and the fact that a cluster tilting subcategory in the stable category lifts to one in the Frobenius cover, when we add all projective-injective objects, which correspond to boundary arcs. 
    For the interested reader we provide a direct computation in the category $\cC_2$ in the appendix. 
\end{proof}

\section{Triangulations and mutations in the Grassmannian category $\cC_2$}

We have seen in Section \ref{S:combinatorial model} that the Grassmannian category $\cC_2$ can be approached via the completed $\infty$-gon. We now explore this combinatorics further, providing a categorical interpretation of triangulations of the completed $\infty$-gon and comparing their categorical and combinatorial mutations.

\subsection{Mutations of triangulations}

To describe mutations within the completed $\infty$-gon, we use the conventions from \cite{BG}, except that we identify the points $+\infty$ and $-\infty$, and just call it $-\infty$ to align with our conventions from Section \ref{S:combinatorial model}. 

\begin{dfn} Let $T$ be a triangulation of the completed $\infty$-gon. An arc $\gamma \in T$ is called \emph{mutable} if there exists $\gamma' \neq \gamma $ such that 
\begin{align*}
    T'= T \backslash \{\gamma\} \cup \{\gamma'\}
\end{align*}
is a triangulation.  We then call the triangulation $T'$ the {\em mutation of $T$ at $\gamma$} and denote it by $\mu_\gamma(T)$.
\end{dfn}

\begin{lemma}[{\cite[Proposition 2.8]{BG}}] \label{lem: mutable arcs}
An arc $\gamma \in T$ is not mutable if and only if it is a wrapping arc.
\end{lemma}
To summarise, there are two types of mutation:
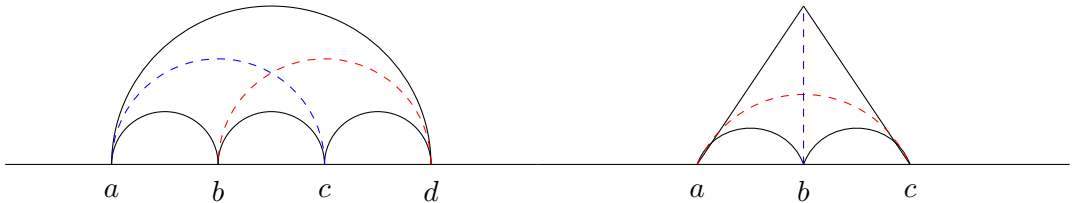
\begin{center}
\begin{tikzpicture}
\node at (-3.5,0) {\begin{tikzpicture}[scale=0.7]
\draw (-5,0) -- (5,0);
\draw (1,0) arc (0:180:1);
\draw (3,0) arc (0:180:1);
\draw (-1,0) arc (0:180:1);
\draw (3,0) arc (0:180:3);
\draw[dashed,red] (3,0) arc (0:180:2);
\draw[dashed,blue] (1,0) arc (0:180:2);
\node at (-3,-0.5) {$a$};
\node at (-1,-0.5) {$b$};
\node at (1,-0.5) {$c$};
\node at (3,-0.5) {$d$};
\end{tikzpicture}};
\node at (3.5,0) {\begin{tikzpicture}[scale=0.7]
\draw (-5,0) -- (5,0);
\draw (-2,0) -- (0,3);
\draw (2,0) -- (0,3);
\draw (0,0) to[looseness=1.25,out=70, in=110] (2,0);
\draw (-2,0) to[looseness=1.25,out=70, in=110] (0,0);
\draw[dashed,red] (-2,0) to[looseness=1.25,out=65, in=115] (2,0);
\draw[dashed,blue] (0,0) to (0,3);
\node at (-2,-0.5) {$a$};
\node at (0,-0.5) {$b$};
\node at (2,-0.5) {$c$};
\end{tikzpicture}};
\end{tikzpicture}
\end{center}

In particular, any mutable arc in a triangulation must belong to one of these two configurations in the triangulation, as either of the dotted arcs. 

Consider now the possible exchange graphs of triangulations of the completed $\infty$-gon with vertices given by triangulations, and edges by mutations. Clearly, if we only consider finitely many mutations, then the exchange graph is not connected. In fact, it has infinitely many connected components. In order to connect the exchange graph, we need to consider infinite sequences of mutations. Indeed, it turns out that we obtain connectedness using a process called {\em transfinite mutations}, see \cite[Definition 6.1]{BG}.

\begin{theorem}[{\cite[Theorem 6.9]{BG}}] \label{thm: connect}
    The exchange graph of triangulations of the completed $\infty$-gon is connected under transfinite mutations.
\end{theorem}

\subsection{Categorifying triangulations}

To make use of transfinite mutations, and in particular Theorem \ref{thm: connect}, we need to understand which subcategories of $\cC_2$ correspond to triangulations of the completed $\infty$-gon, and how to mutate them. However, Proposition \ref{prop:exts} shows that any triangulation with more than one infinite arc is not rigid and so a weaker notion is needed. For this we use maximal almost rigid subcategories of $\cC_2$, the definition of which builds on the definition of maximal almost rigid modules by Barnard, Gunawan, Meehan and Schiffler in \cite{BGMS}.

\begin{dfn} \label{dfn: mar} \begin{enumerate}[leftmargin=0cm,itemindent=.6cm,labelwidth=\itemindent,labelsep=0.1cm,align=left, itemsep=0.2cm]
    \item Two indecomposable modules $M$ and $N$ in $\cC_2$ are called \emph{almost compatible} if they have no non-split extensions, or if all extensions between them have indecomposable middle terms.
    \item A subcategory $\mathcal{A} \subset \cC_2$ is {\em almost rigid} if any two indecomposable modules $M$ and $N$ in $\mathcal{A}$ are almost compatible.
    \item A subcategory $\mathcal{A}$ is \emph{maximal almost rigid} if it is almost rigid and if for every module $M$ not in $\mathcal{A}$, the subcategory $\mathrm{add}(\mathcal{A} \cup M)$ is not almost rigid.
\end{enumerate}
\end{dfn}

\begin{theorem}\label{thm: cat of triang}
    A subcategory $\cA \subset \cC_2$ is maximal almost rigid if and only if its indecomposable objects correspond to a triangulation of the completed $\infty$-gon.
\end{theorem}
\begin{proof}

Consider two indecomposables $M, N \in \cC_2$. If their arcs cross, then the configuration must be the one shown in either Lemma \ref{lem:ses1} or Lemma \ref{lem:ses2}. In each case, the lemma in question shows that there are extensions between them with decomposable middle terms. In other words, no two crossing arcs can correspond to almost compatible modules. 

If $M,N$ do not cross, Proposition \ref{prop:exts} shows that either there are no extensions between them, or they both correspond to infinite arcs and have a one-dimensional extension group in one direction. In the latter case, Lemma \ref{lem:ses3} shows that the only non-split extension has an indecomposable middle term. In other words, any two non-crossing arcs are almost compatible.

This shows two indecomposable modules $M$ and $N$ are almost compatible if and only if their corresponding arcs are non-crossing. The result follows immediately as triangulations are maximal sets of pairwise non-crossing arcs, and maximal almost rigid categories are maximal sets of pairwise almost compatible modules.
\end{proof}

\begin{remark}
   For a subcategory, rigid implies almost rigid, however maximal rigid does not in general imply maximal almost rigid, as the following example illustrates.
   
   Consider the subcategory with indecomposable objects given by arcs in the following picture. 
 
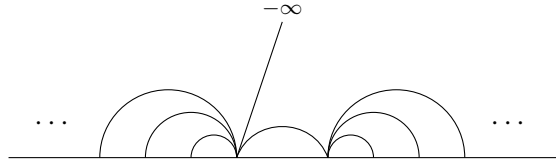
\[\begin{tikzpicture}[scale=0.6]
\draw (-6,0) -- (6,0);
\draw (0,3) -- (-1,0);
\draw (-1,0) to[looseness=1.25,out=70, in=110] (1,0);
\draw (-1,0) arc (0:180:1);
\draw (-1,0) arc (0:180:1.5);
\draw (-1,0) arc (0:180:0.5);
\draw (3,0) arc (0:180:1);
\draw (4,0) arc (0:180:1.5);
\draw (2,0) arc (0:180:0.5);
\node at (5,0.75) {$\cdots$};
\node at (-5,0.75) {$\cdots$};
\node at (0,3.3) {$\scriptstyle -\infty$};
\end{tikzpicture}\]
   
   This subcategory is maximal rigid, but not maximal almost rigid: We could add the wrapping arc connecting the source of the right fountain with $-\infty$, which is almost compatible with all the depicted arcs.
\end{remark}

\subsection{Mutation}

We are now going to define mutation of almost rigid subcategories in analogy to the mutation in triangulated categories of Iyama and Yoshino \cite{IYmutation}.

\begin{dfn} \label{mutationdefinition}
    Let $\cA$ be an almost rigid subcategory of an exact category $\cC$. We call an indecomposable object $X$ of $\cA$ {\em mutable} if there exists both a left $\add(\cA \setminus X)$-approximation and a right $\add(\cA \setminus X)$-approximation of $X$. 
    
    In that case, we define
    \begin{align*}
        \mu^-_X(\cA) = \add \{Z \in \cC  \mid & \ \text{there exists an exact sequence } 0 \to A' \xrightarrow{f} A'' \to Z \to 0 \\ 
        & \text{ such that $A' \in \cA$ and $f$ is a left $\add(\cA \setminus X)$-approximation of $A'$} \}
    \end{align*}
and call this the {\em left mutation} of $\cA$ at $X$. Dually we define
    \begin{align*}
        \mu^+_X(\cA) = \add \{Z \in \cC  \mid & \ \text{there exists an exact sequence } 0 \to Z \xrightarrow{} A'' \xrightarrow{g} A' \to 0 \\ 
        & \text{ such that $A' \in \cA$ and $g$ is a right $\add(\cA \setminus X)$-approximation of $A'$} \}
    \end{align*}
and call this the {\em right mutation} of $\cA$ at $X$.

We call the short exact sequences appearing in the definition of $\mu^-$ and $\mu^+$ {\em exchange sequences}. 
\end{dfn}
If $X$ is mutable, left mutation corresponds to simply replacing the indecomposable $X$ with the indecomposable $Z$ such that there is an exchange sequence
    \begin{align*}
        0 \to X \to A' \to Z \to 0
    \end{align*}
and similarly for right mutation.
\begin{remark} \begin{enumerate}[leftmargin=0cm,itemindent=.6cm,labelwidth=\itemindent,labelsep=0.1cm,align=left, itemsep=0.2cm]
    \item The subcategory $\add(\cA \setminus X)$ is always almost rigid, but not in general rigid. As such, our definition extends the framework of mutation defined in \cite{IYmutation}.
    \item Furthermore, the short exact sequences in our definition of mutation almost mirror the exchange sequences for (weak) cluster structures as introduced by Buan, Iyama, Reiten and Scott in \cite{BIRS}. However, we do not have the clear cut distinction between coefficients and cluster variables: We have indecomposable objects that can show up as a mutable indecomposable in one almost rigid subcategory, but as a non-mutable indecomposable in another almost rigid subcategory, see Example \ref{Ex: mutable arcs}. In the language of \cite{BIRS} this would correspond to the indecomposable object in question to be a coefficient in one cluster, and a cluster variable in another.
    \item We will see in Example \ref{Ex: why both sides} why we insist both a left and right $\add(\cA \setminus X)$-approximation of $X$ exists. Indeed, there we consider an indecomposable $X$ with only a left $\add(\cA \setminus X)$-approximation and note that neither $\mu^-_X(\cA)$ or $\mu^+_X(\cA)$ are maximal almost rigid.
    \end{enumerate}
    \end{remark}

We now describe the mutable indecomposable objects in the almost rigid subcategories of $\cC_2$ in terms of the combinatorial model.

\begin{theorem} \label{t: mutation}
    Let $\cA$ be a maximal almost rigid subcategory of $\cC_2$ corresponding to a triangulation $T$ of the completed $\infty$-gon. An indecomposable object $X$ of $\cA$ is mutable if and only if the corresponding arc $\gamma$ in $T$ is mutable. Furthermore, the left and right mutation of $\cA$ at $X$ coincide, and correspond to the mutation $\mu_\gamma(T)$ of $T$ at $\gamma$.
\end{theorem}

\begin{proof}
Suppose $\gamma$ is a mutable arc in some triangulation $T$, and $X$ is the corresponding object in the corresponding maximal almost rigid subcategory $\cA$. Since $\gamma$ is mutable, $T$ must contain a configuration such as in Lemma \ref{lem:ses1} or \ref{lem:ses2}, where $\gamma$ is one of the dotted arcs. In particular, suppose $\gamma=(a,b)$ as in Lemma \ref{lem:ses1} so that $\mu_\gamma(T)=T \backslash (a,b) \cup (c,d)$, and consider the two exact sequences 
\begin{align*}
    0 \to (a,b) \xrightarrow{f} (c,b) \oplus (a,d) \xrightarrow{g} (c,d) \to 0\\
    0 \to (c,d) \xrightarrow{f'} (a,c) \oplus (b,d) \xrightarrow{g'} (a,b) \to 0
\end{align*}
from Lemma \ref{lem:ses1}. We claim that $f$ (resp.\ $g'$) is a left (resp.\ right) $\add(\cA\setminus X)$-approximation of $X$, and hence $X$ is mutable with both $\mu_X^+(\cA)$ and $\mu_X^-(\cA)$ corresponding to the triangulation $\mu_\gamma(T)$.

Indeed, to show $f$ is a left $\add(\cA\setminus X)$-approximation, it is enough to show 
\begin{align*}
    \Ext^1\left((c,d\right),\gamma)=0
\end{align*}
for all $\gamma \in T\backslash (a,b)$. Since $T\backslash (a,b) \cup (c,d)$ is a triangulation, no arc $\gamma \in T \backslash (a,b)$ can cross $(c,d)$ and so this follows from Proposition \ref{prop:exts}. Similarly, to show $g'$ is a right $\add(\cA\setminus X)$-approximation, it is enough to show 
\begin{align*}
    \Ext^1(\gamma,(c,d))=0
\end{align*}
for all $\gamma \in T\backslash (a,b)$. Since $T\backslash (a,b) \cup (c,d)$ is a triangulation, no arc $\gamma \in T\backslash (c,d)$ can cross $(a,b)$ and so this follows from Proposition \ref{prop:exts}.

The cases for $\gamma=(c,d)$ in Lemma \ref{lem:ses1}, and the two cases in Lemma \ref{lem:ses2} are all similar. \\

So we now show that if an arc $\gamma$ is not mutable then the corresponding object $X$ is not mutable. If $\gamma$ is not mutable, then it is a wrapping arc by \cite[Proposition 2.8]{BG}. Assume that $\gamma$ is a wrapping arc for a left fountain at $n$. The case for a right fountain follows symmetrically. 

Since $\gamma$ is a wrapping arc, there exist infinitely many finite arcs $(m,n) \in T$ with $m<n$. Moreover, for any such $(m,n)$, Lemma \ref{L:infinite to finite} and the comments thereafter show that there is a nonzero morphism 
\[
f \colon (-\infty,n) \to (m,n)
\]
where $1 \mapsto x$. If this were to factor through another indecomposable $(s,t) \in T$, then we would have nonzero maps
\[
(-\infty,n) \to (s,t) \quad \text{and} \quad (s,t) \to (m,n).
\]
In particular, if $s=-\infty$, Lemmas \ref{L: infinite to infinite} and \ref{L:infinite to finite} show that $n \leq t \leq n$ and hence $(s,t)=\gamma$. Thus $f$ does not factor through any infinite arc in $T \backslash \gamma$.

If $(s,t)$ is a finite arc, then $n \leq t$ and the map $(-\infty,n) \to (s,t)$ is determined by $1 \mapsto \alpha xy^{t-n}$ for some $\alpha \in \C$ by Lemma \ref{L:infinite to finite}. Since $f$ is nonzero, the map $(s,t) \to (m,n)$ cannot send $x$ to 0, and so Lemma \ref{L:finite to finite} shows that $s \leq m$ and $n \leq t \leq n$. So any left $\add(\cA\setminus X)$-approximation of $\gamma$ must be a finite direct sum which, for each $(m,n) \in T$ with $m < n$, contains some $(s,n) \in T \backslash \gamma$ with $s \leq m$. But there are infinitely many such $(m,n) \in T$ and so this is not possible and hence $\gamma$ has no left $\add(\cA\setminus X)$-approximation as required.
\end{proof}

  \begin{ex} \label{Ex: mutable arcs} In the first triangulation, the infinite arc $(-\infty,0)$ is mutable as it can be replaced with the arc $(-1,1)$. In the second triangulation, the arc $(-\infty,0)$ is not mutable.
 
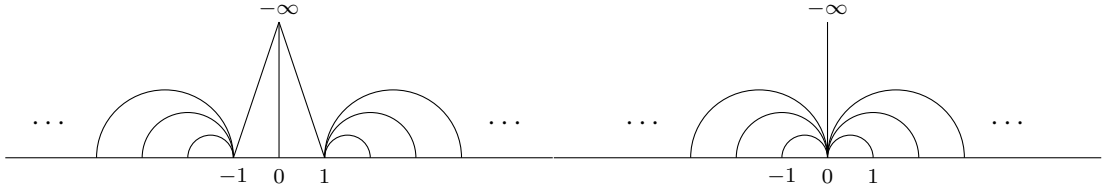
\[\begin{tikzpicture}[scale=0.6]
\draw (-6,0) -- (6,0);
\draw (0,3) -- (-1,0);
\draw (0,3) -- (0,0);
\draw (0,3) -- (1,0);
\draw (-1,0) arc (0:180:1);
\draw (-1,0) arc (0:180:1.5);
\draw (-1,0) arc (0:180:0.5);
\draw (3,0) arc (0:180:1);
\draw (4,0) arc (0:180:1.5);
\draw (2,0) arc (0:180:0.5);
\node at (5,0.75) {$\cdots$};
\node at (-5,0.75) {$\cdots$};
\node at (0,3.3) {$\scriptstyle -\infty$};
\node at (0,-0.4) {$\scriptstyle 0$};
\node at (-1,-0.4) {$\scriptstyle -1$};
\node at (1,-0.4) {$\scriptstyle 1$};
\end{tikzpicture}
\begin{tikzpicture}[scale=0.6]
\draw (-6,0) -- (6,0);
\draw (0,3) -- (0,0);
\draw (0,0) arc (0:180:1);
\draw (0,0) arc (0:180:1.5);
\draw (0,0) arc (0:180:0.5);
\draw (2,0) arc (0:180:1);
\draw (3,0) arc (0:180:1.5);
\draw (1,0) arc (0:180:0.5);
\node at (4,0.75) {$\cdots$};
\node at (-4,0.75) {$\cdots$};
\node at (0,3.3) {$\scriptstyle -\infty$};
\node at (0,-0.4) {$\scriptstyle 0$};
\node at (-1,-0.4) {$\scriptstyle -1$};
\node at (1,-0.4) {$\scriptstyle 1$};
\end{tikzpicture}\]
\end{ex}

\begin{ex} \label{Ex: why both sides}
 Consider the arc $\gamma=(-\infty,-1)$ in the first triangulation of Example \ref{Ex: mutable arcs}. The proof of Theorem \ref{t: mutation} shows that $\gamma$ has no left $\add(\cA\setminus X)$-approximation, so if we naively extend Iyama-Yoshino mutation, the left mutation would be 
  
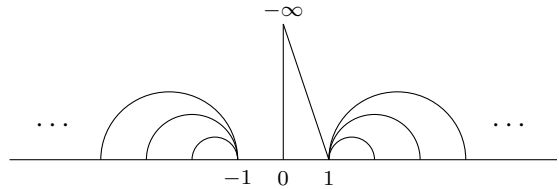
\[\begin{tikzpicture}[scale=0.6]
\draw (-6,0) -- (6,0);
\draw (0,3) -- (0,0);
\draw (0,3) -- (1,0);
\draw (-1,0) arc (0:180:1);
\draw (-1,0) arc (0:180:1.5);
\draw (-1,0) arc (0:180:0.5);
\draw (3,0) arc (0:180:1);
\draw (4,0) arc (0:180:1.5);
\draw (2,0) arc (0:180:0.5);
\node at (5,0.75) {$\cdots$};
\node at (-5,0.75) {$\cdots$};
\node at (0,3.3) {$\scriptstyle -\infty$};
\node at (0,-0.4) {$\scriptstyle 0$};
\node at (-1,-0.4) {$\scriptstyle -1$};
\node at (1,-0.4) {$\scriptstyle 1$};
\end{tikzpicture}
\]
which is not a triangulation/maximal almost rigid. This is somewhat expected as approximations are key to mutation. However, what is perhaps more surprising is that even though a right approximation does exist (it is given by $(-\infty, 0)$), the result under right mutation is the same as the left mutation and so is not maximal almost rigid. This is why we restrict mutation to objects which have both a left and right approximation.
 
 \end{ex}

 By virtue of Theorem \ref{t: mutation} the concepts of infinite, completed and transfinite mutations of the completed $\infty$-gon from \cite{BG} can be directly extended to the mutation of maximal almost rigid subcategories of $\cC_2$.
 
 We can define the exchange graph of maximal almost rigid subcategories of $\cC_2$ as the graph with vertices corresponding to maximal almost rigid subcategories, and with edges given by transfinite mutations.
 
\begin{cor} \label{c: connected}
    The exchange graph of maximal almost rigid subcategories of $\cC_2$ is connected.
\end{cor}

Since our notion of mutation restricts to Iyama-Yoshino mutation in the cluster tilting setting, Theorem \ref{t: mutation} also shows that the mutation of these subcategories (where possible) is controlled by the combinatorics of the completed $\infty$-gon.

In summary, we see that the category $\cC_2$ arising from the $A_\infty$ curve singularity naturally exhibits cluster combinatorics induced from both the $\infty$-gon and its completion. 

\appendix \label{appendix}

\section{} \label{appendix:A}

We continue to use our convention of identifying indecomposable objects in $\CC_2 = \CM^\Z(R)$ by arcs in the completed $\infty$-gon.

\subsection{Hom-calculations}\label{S:Hom-calculations}

We describe the homomorphisms in $\cC_2$. Overall, we have the following description of $\Hom$-spaces:

\begin{proposition}\label{P:Homspaces}
Let $(a,b)$ and $(c,d)$ be indecomposable objects in $\cC_2$. Then
\[
    \Hom_R((a,b),(c,d)) \cong \begin{cases}
                            \C^2 & \text{if $-\infty < a \leq c$ and $b \leq d$}\\
                            0 & \text{if $-\infty = a \leq c$ and $d<b$}\\
                            0 & \text{if $d<a$}\\
                            \C & \text{else.}
                                            \end{cases}
\]
\end{proposition}

We now tackle Proposition \ref{P:Homspaces} case by case, and give an explicit basis of the Hom-space in each case. Recall from Section \ref{S:model} that the module $(a,b)$ is, up to isomorphism, given by
$(x,y^{b-a-1})(-b+1)$ if $(a,b)$ is finite
and $\C[y](-b)$  if $(a,b)$ is infinite.

\begin{lemma} \label{L: infinite to infinite} The spaces of homomorphisms between infinite arcs are given by:
\begin{align}
    \Hom_R \left( \C[y](j), \C[y] \right)= 
    \begin{cases} 
    \C \qquad \text{ if } j \geq 0 \\
    0 \qquad \text{otherwise.}
    \end{cases}
\end{align}
In the case where there is a map, $1 \mapsto \lambda y^j$ for some $\lambda \in \C$.
\end{lemma}

\begin{proof}
Any such morphism is determined by where $1 \in \C[y](j)$ is mapped to, and to be a degree zero morphism, it must map to an element of degree $-j$. If $j < 0$, the only possibility is $0$, and when $j \geq 0$, this is $\lambda y^j$ for $\lambda \in \C$.
\end{proof}

We can restate Lemma \ref{L: infinite to infinite} as follows: Let $(-\infty,b)$ and $(-\infty,d)$ be indecomposable objects in $\CM^\Z(R)$. Then
\[
    \Hom_R((-\infty,b),(-\infty,d)) \cong \begin{cases}
                                                \C & \text{if $b \leq d$}\\
                                                0 & \text{else}
                                            \end{cases}
\]
where any existing map is determined by $1 \mapsto \lambda y^{d-b}$ for some $\lambda \in \C$. The case where non-trivial morphisms exist is depicted in Figure \ref{fig:infinite to infinite}.
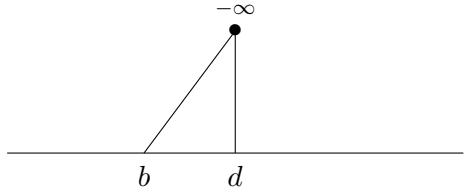
\begin{figure}[H]
\centering

\begin{tikzpicture}[scale=0.6]
\draw (-5,0) -- (5,0);
\draw (-2,0) -- (0,2.7);
\draw (0,0) -- (0,2.7);
\node at (-2,-0.5) {$b$};
\node at (0,-0.5) {$d$};
\node at (0,3.2) {$\scriptstyle -\infty$};
\node at (0, 2.7) {$\bullet$};
\end{tikzpicture}   
 \caption{$  \Hom_R((-\infty,b),(-\infty,d)) \cong  \C$.}
 \label{fig:infinite to infinite}
\end{figure}

\begin{lemma} \label{L: finite to infinite} The spaces of homomorphisms from a finite arc to an infinite arc are given by:
\begin{align}
    \Hom_R \left( (x,y^i)(j), \C[y] \right) \cong 
    \begin{cases} 
    \C \qquad \text{ if } j \geq -i \\
    0 \qquad \text{otherwise.}
    \end{cases}
\end{align}
In the case where there is a map, $x \mapsto 0$ and $y^i \mapsto \lambda y^{i+j}$ for some $\lambda \in \C$.
\end{lemma}

\begin{proof}
Any such morphism $g$ is determined by where $x$ and $y^i$ are mapped to. Moreover, as $x$ annihilates $\C[y] = R/(x)$, $g$ must satisfy
\[
y^ig(x)=xg(y^i)=0
\]
and thus $g(x)=0$. Now, to be a degree zero morphism $y^i$ must map to an element of degree $-i-j$, which is precisely one of the form $ \lambda y^{i+j}$ if $j \geq -i$ and zero otherwise.
\end{proof}

Lemma \ref{L: finite to infinite} can be restated as follows:  Let $(a,b)$ and $(-\infty,d)$ be indecomposable objects in $\CM^\Z(R)$, with $(a,b)$ a finite arc. Then
\[
    \Hom_R((a,b),(-\infty,d)) \cong \begin{cases}
                                \C & \text{if $a \leq d$}\\
                                                0 & \text{else}
                                            \end{cases}
\]
where any existing map is determined by $x \mapsto 0$ and $y^{b-a-1} \mapsto\lambda y^{d-a}$ for some $\lambda \in \C$.
In other words, there are morphisms if the finite arc starts at or to the left of the infinite arc, as depicted in Figure \ref{fig:finite to infinite}.

\begin{figure}[H]
\begin{center}
\begin{tikzpicture}
\node at (-3,-.5) {
\begin{tikzpicture}[scale=0.6]
\draw (-4,0) -- (4,0);
\draw (0,2.7) -- (0,0);
\draw (-1,0) arc (0:180:.75);
\node at (-2.7,-.5) {$a$};
\node at (-1,-.5) {$b$};
\node at (0,3.2) {$\scriptstyle -\infty$};
\node at (0, 2.7) {$\bullet$};
\node at (0,-.5) {$d$};
\end{tikzpicture}};
\node at (3,-.5) {
\begin{tikzpicture}[scale=0.6]
\draw (-4,0) -- (4,0);
\draw (0,2.7) -- (0,0);
\draw (1,0) arc (0:180:1.5);
\node at (-2,-.5) {$a$};
\node at (1.25,-.5) {$b$};
\node at (0,3.2) {$\scriptstyle -\infty$};
\node at (0, 2.7) {$\bullet$};
\node at (0,-.5) {$d$};
\end{tikzpicture}};
\end{tikzpicture}
\end{center}
  \caption{$ \Hom_R((a,b),(-\infty,d))\cong   \C$.}   
  \label{fig:finite to infinite}
 \end{figure}
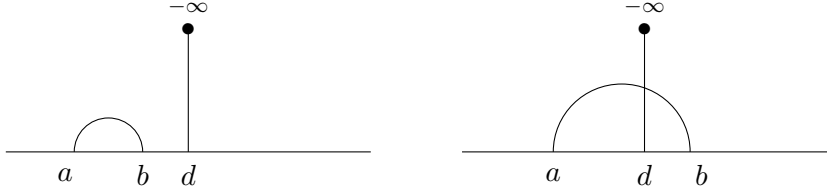

\begin{lemma} \label{L:infinite to finite}The space of homomorphisms from an infinite arc to a finite arc is given by:
\begin{align}
    \Hom_R \left( \C[y](j), (x,y^i) 
    \right) \cong
    \begin{cases} 
    \C \qquad \text{ if } j \geq -1 \\
    0 \qquad \text{otherwise.}
    \end{cases}
\end{align}
In the case where there is a map, $1 \mapsto \lambda xy^{j+1}$ for some $\lambda \in \C$.
\end{lemma}

\begin{proof}
Any such morphism $g$ is determined by where $1 \in \C[y](j)$ is mapped to. Moreover, as $x$ annihilates $\C[y]=R/(x)$ then $xg(1)=g(x \cdot 1)=0$, and thus $x$ must map to an element in $(x)$. If $j<-1$, then the only option is zero. If $j \geq -1$ then $1 \mapsto \lambda xy^{j+1}$ for some $\lambda \in \C$.
\end{proof}

Lemma \ref{L:infinite to finite} can be restated as follows: Let $(-\infty,b)$ and $(c,d)$ be indecomposable objects in $\CM^\Z(R)$, with $(c,d)$ a finite arc. Then
\[
    \Hom_R((-\infty,b),(c,d)) \cong \begin{cases}
                                                \C & \text{if $b \leq d$}\\
                                                0 & \text{else},
                                            \end{cases}
\]
where any existing map is determined by $1 \mapsto \lambda xy^{d-b}$ for some $\lambda \in \C$.
In other words, there are morphisms if the finite arc ends at or to the right of the infinite arc, as shown in Figure \ref{fig:infinite to finite}.

\begin{figure}[H]
\begin{center}
\begin{tikzpicture}
\node at (3,-.5) {
\begin{tikzpicture}[scale=0.6]
\draw (-4,0) -- (4,0);
\draw (0,2.7) -- (0,0);
\draw (2.5,0) arc (0:180:.75);
\node at (1.05,-.5) {$c$};
\node at (2.5,-.5) {$d$};
\node at (0,3.2) {$\scriptstyle -\infty$};
\node at (0, 2.7) {$\bullet$};
\node at (0,-.5) {$b$};
\end{tikzpicture}};
\node at (-3,-.5) {
\begin{tikzpicture}[scale=0.6]
\draw (-4,0) -- (4,0);
\draw (0,2.7) -- (0,0);
\draw (1,0) arc (0:180:1.5);
\node at (-2,-.5) {$c$};
\node at (1.25,-.5) {$d$};
\node at (0,3.2) {$\scriptstyle -\infty$};
\node at (0, 2.7) {$\bullet$};
\node at (0,-.5) {$b$};
\end{tikzpicture}};
\end{tikzpicture}
\end{center}
  \caption{$\Hom_R((-\infty,b),(c,d)) \cong \C$.}   
  \label{fig:infinite to finite}
 \end{figure}
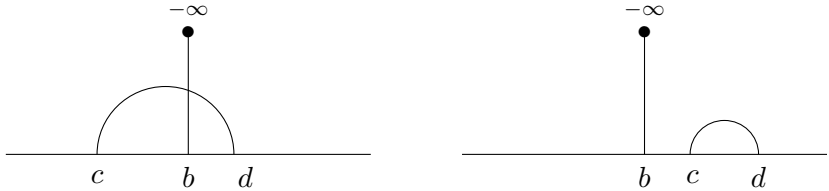

\begin{lemma}\label{L:finite to finite}
The spaces of homomorphisms between finite arcs are given by
\[
     \Hom_R((a,b),(c,d))
                 \cong \begin{cases}
                                \C^2 & \text{if $a \leq c$ and $b \leq d$}\\
                                \C & \text{if $a \leq c$ and $d<b$ or $c<a \leq d$}\\
                                0 & \text{else; i.e.\ if $d<a$.}\\
                                          \end{cases}
\]
\end{lemma}
\begin{proof}
Recall that a map from $(a,b)$ to $(c,d)$ is precisely a degree zero homomorphism
\[
g \colon (x,y^{b-a-1})(1-b) \to (x,y^{d-c-1})(1-d)
\]
and any such map is determined by $g(x)$ and $g(y^{b-a-1})$. Since $g$ is degree preserving, there must exist $\alpha, \beta, \gamma, \delta \in \bC$ such that $g(x)=\alpha xy^{d-b}+ \beta y^{d-b-1}$ and $g(y^{b-a-1})=\gamma xy^{d-a}+\delta y^{d-a-1}$. Moreover,
    \[
       y^{b-a-1}(\alpha xy^{d-b} + \beta y^{d-b-1})=y^{b-a-1}g(x) =g(xy^{b-a-1}) = xg(y^{b-a-1}) = \delta x y^{d-a-1}
    \]
and thus $\beta=0$ and $\alpha=\delta$. It follows that the Hom space is at most two-dimensional.

Note that $\gamma$ can be nonzero if and only if $xy^{d-a} \in (x,y^{d-c-1})(1-d)$ which is if and only if $a \leq d$. Further, $\alpha=\delta$ can be nonzero if and only if both $xy^{d-b}$ and $y^{d-a-1}$ lie in $(x,y^{d-c-1})(1-d)$. The first is satisfied if and only if $b \leq d$ while the second holds if and only if $d-a-1 \geq d-c-1$, or equivalently, $a \leq c$. 

It follows that the morphisms can be described as follows:
\begin{enumerate}
    \item If $a \leq d$ then there are maps determined by $x \mapsto 0, y^{b-a-1} \mapsto \gamma xy^{d-a}$ for each $\gamma \in \C$.
    \item If $a \leq c$ and $b \leq d$ then for each $\alpha \in \C$, there is a map taking $x \mapsto \alpha xy^{d-b}$ and $y^{b-a-1} \mapsto \alpha y^{d-a-1}$.
\end{enumerate}
Notice that if condition $(2)$ is satisfied, then so is $(1)$, and these are precisely the conditions for the space of homomorphisms to be two-dimensional in the above. When $(1)$ is satisfied but $(2)$ is not this gives the case where the space of homomorphisms is one-dimensional, and when neither are satisfied, all homomorphisms are $0$.
\end{proof}

Some of the cases from Lemma \ref{L:finite to finite} are depicted in Figures \ref{fig:Case1}-\ref{fig:Case3}.

\begin{figure}[H]
\begin{center}
\begin{tikzpicture}
\node at (-3,0) {
\begin{tikzpicture}[scale=0.6]
\draw (-4,0) -- (4,0);
\draw (-1,0) arc (0:180:1.3);
\draw (2,0) arc (0:180:.8);
\node at (-3.6,-.5) {$a$};
\node at (-1,-.5) {$b$};
\node at (.4,-.5) {$c$};
\node at (2,-.5) {$d$};
\end{tikzpicture}};
\node at (3,0) {
\begin{tikzpicture}[scale=0.6]
\draw (-4,0) -- (4,0);
\draw (0,0) arc (0:180:1.5);
\draw (2,0) arc (0:180:1.5);
\node at (-3,-.5) {$a$};
\node at (0,-.5) {$b$};
\node at (-1,-.5) {$c$};
\node at (2,-.5) {$d$};
\end{tikzpicture}};
\end{tikzpicture}
\end{center}
  \caption{$\Hom_R((a,b),(c,d))   \cong  \C^2$.}  
  \label{fig:Case1}
 \end{figure}
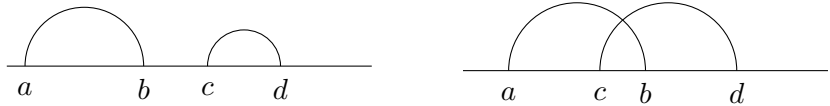

\begin{figure}[H]
\begin{center}
\begin{tikzpicture}
\node at (-4,0) {
\begin{tikzpicture}[scale=0.45]
\draw (-3,0) -- (4,0);
\draw (3,0) arc (0:180:2);
\draw (1.5,0) arc (0:180:.8);
\node at (-1,-.5) {$a$};
\node at (3,-.5) {$b$};
\node at (1.5,-.5) {$d$};
\node at (-.1,-.5) {$c$};
\end{tikzpicture}};
\node at (0,0) {
\begin{tikzpicture}[scale=0.45]
\draw (-3,0) -- (4,0);
\draw (3,0) arc (0:180:2);
\draw (1,0) arc (0:180:.5);
\node at (-1,-.5) {$c$};
\node at (3,-.5) {$d$};
\node at (1,-.5) {$b$};
\node at (-.1,-.5) {$a$};
\end{tikzpicture}};
\node at (4,0) {
\begin{tikzpicture}[scale=0.45]
\draw (-4,0) -- (4,0);
\draw (0,0) arc (0:180:1.5);
\draw (2,0) arc (0:180:1.5);
\node at (-3,-.5) {$c$};
\node at (0,-.5) {$d$};
\node at (-1,-.5) {$a$};
\node at (2,-.5) {$b$};
\end{tikzpicture}};
\end{tikzpicture}
\end{center}
  \caption{$\Hom_R((a,b),(c,d)) \cong  \C$.}
  \label{fig:Case2}
\end{figure}
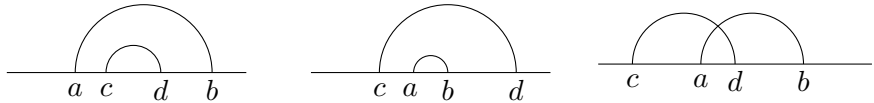

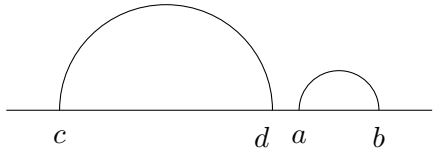
\begin{figure}[H]
\begin{center}
\begin{tikzpicture}[scale=0.7]
\draw (-4,0) -- (4,0);
\draw (1,0) arc (0:180:2);
\draw (3,0) arc (0:180:.75);
\node at (-3,-.5) {$c$};
\node at (.8,-.5) {$d$};
\node at (1.5,-.5) {$a$};
\node at (3,-.5) {$b$};
\end{tikzpicture}
\end{center}
      \caption{$\Hom_R((a,b),(c,d))  =  0$.}   
      \label{fig:Case3}
\end{figure}

\subsection{Cluster tilting subcategories}

If $T$ is a cluster tilting subcategory of $\cC_2$ then it is also maximal rigid. By Proposition \ref{prop:exts}, its indecomposable objects must therefore correspond to a maximal set of mutually non-crossing arcs containing at most one infinite arc.

As a consequence, the maximal rigid subcategories $\mc{T}$ of $\cC_2$ are of the following form:
\begin{enumerate}
\item $\mc{T}$ corresponds to a triangulation of the completed $\infty$-gon that is locally finite. In this case, it contains only finite arcs. 
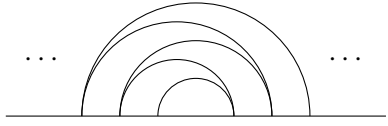
\begin{center}
\begin{tikzpicture}[scale=0.5]
\draw (-5,0) -- (5,0);
\draw (1,0) arc (0:180:1);
\draw (1,0) arc (0:180:1.5);
\draw (2,0) arc (0:180:2);
\draw (2,0) arc (0:180:2.5);
\draw (3,0) arc (0:180:3);
\node at (4,1.5) {$\cdots$};
\node at (-4,1.5) {$\cdots$};
\end{tikzpicture}
\end{center}

\item $\mc{T}$ corresponds to a maximal set of non-crossing finite arcs containing a split fountain, i.e., a left fountain at $a$ and a right fountain at $b$ with $a < b$ together with a unique infinite arc given by either $(-\infty,a)$ or $(-\infty,b)$.

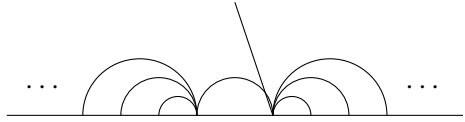
\begin{center}
    \begin{tikzpicture}[scale=0.5]
\draw (-6,0) -- (6,0);
\draw (1,0) arc (0:180:1);
\draw (0,3) -- (1,0);
\draw (-1,0) arc (0:180:1);
\draw (-1,0) arc (0:180:1.5);
\draw (-1,0) arc (0:180:0.5);
\draw (3,0) arc (0:180:1);
\draw (4,0) arc (0:180:1.5);
\draw (2,0) arc (0:180:0.5);
\node at (5,0.75) {$\cdots$};
\node at (-5,0.75) {$\cdots$};
\end{tikzpicture}
\end{center}
\item $\mc{T}$ corresponds to a triangulation of the completed $\infty$-gon containing a fountain at $a$, and thus also containing the arc $(-\infty,a)$.

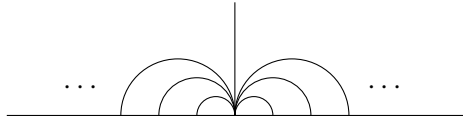
\begin{center}
    \begin{tikzpicture}[scale=0.5]
\draw (-6,0) -- (6,0);
\draw (0,3) -- (0,0);
\draw (0,0) arc (0:180:1);
\draw (0,0) arc (0:180:1.5);
\draw (0,0) arc (0:180:0.5);
\draw (2,0) arc (0:180:1);
\draw (3,0) arc (0:180:1.5);
\draw (1,0) arc (0:180:0.5);
\node at (4,0.75) {$\cdots$};
\node at (-4,0.75) {$\cdots$};
\end{tikzpicture}
\end{center}
\end{enumerate}

We will now provide an alternative proof of Theorem \ref{t: complete ct} using the calculations from Appendix \ref{S:Hom-calculations}. Note first that Case (2) in the above list of maximal rigid subcategories is not cluster tilting: Assume that the infinite arc $(-\infty,a)$ lies in $\cT$. Then $(-\infty,b)$ lies in $\{M \in \CC \mid \Ext^1_\CC(M, \T)=0\}$, but not in $\cT$, which thus is not cluster tilting. The case where $(-\infty,b) \in \cT$ follows symmetrically. We now rule out Case (1). 


\begin{proposition} \label{P: leapfrog not cluster tilting}
Let $\mathcal{T}$ in $\cC_2$ be a maximally rigid subcategory given by a locally finite triangulation $T$. Then $\mathcal{T}$ is not precovering, and thus not functorially finite.
\end{proposition}

\begin{proof}

Let $\gamma$ be an infinite arc. In particular, $\gamma$ is not in $T$. Then by Lemma \ref{L: finite to infinite} and \cite[Lemma 3.7]{GG1} there exist infinitely many arcs $\alpha_i \in T$ such that $\Hom(\alpha_i,\gamma) \neq 0$. Assume for a contradiction that there is a (minimal) $\mathcal{T}$-precover $f \colon \alpha \to \gamma$, where $\alpha = \sum_{i = 1}^n \alpha_i$ for $\alpha_i \in T$.
Assume that $\alpha_n =(a,b) $ is the longest arc in $\{\alpha_1, \ldots, \alpha_n\}$, that is, $b-a$ is maximal.

Now, by \cite[Lemma 3.7]{GG1} and since $T$ is locally finite there exists an arc $\beta = (c,d) \in T$ such that $c < a < b < d$. Then there exists a non-zero morphism from $(c,d)$ to $\gamma$ by Lemma \ref{L: finite to infinite}, since there are morphisms from $\alpha_i$ to $\gamma$ for all $1 \leq i \leq n$. For any map $h$ from $\beta$ to an $\alpha_i$ the image lies in the ideal generated by $xy$ (up to some shift). But then the composition $f \circ h$ has to be $0$. 

So any maximal rigid category containing a leapfrog is not precovering, and therefore not cluster tilting.
\end{proof}

Let $\mathcal{C}$ be any Frobenius category. In order to apply results from \cite{HJ-cat} directly, we first observe that we can lift precovers, and symmetrically preenvelopes, from the stable category $\underline{\mathcal{C}}$ to $\mathcal{C}$. We denote by
\[
    \pi \colon \mathcal{C} \to \underline{\mathcal{C}}
\]
the canonical projection functor.
\begin{lemma} \label{l:lifting covers}
    Let $\mathcal{T}$ be a subcategory of the stable category $\underline{\mathcal{C}}$ containing all projective-injective objects. Let $M$ be an object in $\mathcal{C}$, and assume $\pi(M)$ has a $\pi(\mathcal{T})$-precover (respectively $\pi(\mathcal{T})$-preenvelope) in $\underline{\mathcal{C}}$. Then $M$ has a $\mathcal{T}$-precover (respectively $\mathcal{T}$-preenvelope) in $\mathcal{C}$.
\end{lemma}
    
\begin{proof}
    We only show the claim for precovers, the proof for preenvelopes follows symmetrically. Assume $\pi(f) \colon \pi(T) \to \pi(M)$ is a $\pi(\mathcal{T})$-precover in $\underline{\mathcal{C}}$. The map $\pi(f)$ is induced by a map $f \colon T \to M$ in $\mathcal{C}$. Since our ring $R$ is Gorenstein, every object in $\mathcal{C}$ has a projective precover. Let $p \colon P_M \to M$ be a projective precover of $M$.
    
    We show that the map 
    \[
        \begin{bmatrix}f & p \end{bmatrix} \colon T \oplus P_M \to M
    \]
    is a $\mathcal{T}$-precover of $M$ in $\mathcal{C}$. Since $\mathcal{T}$ contains all projective-injective objects, $T \oplus P_M$ is an object in $\mathcal{T}$. Suppose now we have an object $T'$ in $\mathcal{T}$ with a map $g \colon T' \to M$. Since $\pi(f)$ is a $\pi(\mathcal{T})$-precover in $\underline{\mathcal{C}}$ there exists a map $\pi(h) \colon \pi(T') \to \pi(T)$ such that the following diagram in $\underline{\mathcal{C}}$ commutes:
    \[
        \xymatrix{& \pi(T') \ar[d]^-{\pi(g)} \ar[ld]_-{\pi(h)} \\
                    \pi(T) \ar[r]_-{\pi(f)}&  \pi(M).}
    \]
    Consider the following lift of this diagram in $\mathcal{C}$
    \[
        \xymatrix{& T' \ar[d]^-{g} \ar[ld]_-{\begin{bmatrix}h \\0 \end{bmatrix}} \\
                    T \oplus P_M \ar[r]_-{\begin{bmatrix}f & p \end{bmatrix}}&  M.}
    \]
    We have 
    \[
        \begin{bmatrix}f & p\end{bmatrix}\begin{bmatrix}h \\ 0\end{bmatrix} = fh.
    \]
Now $\pi(g) = \pi(fh)$, so $g = fh + \delta$, for some $\delta \colon T' \to M$ factoring through a projective-injective object $Q$ in $\mathcal{C}$. We have the commutative diagram
\[
    \xymatrix{T'\ar[d] \ar[r]^-{\delta} & M \\ Q \ar[ru] \ar@{-->}[r] & P_M \ar[u]_-p,}
\]
where the dashed arrow exists since $p \colon P_M \to M$ is a projective precover. Therefore $\delta$ factors through $P_M$
    \[
        \xymatrix{T' \ar[rr]^-{\delta} \ar[rd]_-{\delta'} & & M \\
                    & P_M \ar[ru]_-p&.}
    \]
Consider the commutative diagram
    \[
        \xymatrix{& T' \ar[d]^-{g} \ar[ld]_-{\begin{bmatrix}h \\ \delta' \end{bmatrix}} \\
                    T \oplus P_M \ar[r]_-{\begin{bmatrix}f & p \end{bmatrix}}&  M.}
    \]
in $\mathcal{C}$. We have
    \[
        \begin{bmatrix}f & p\end{bmatrix}\begin{bmatrix}h \\ \delta' \end{bmatrix} = fh + p\delta' = fh + \delta = g.
    \]
Therefore, $g$ factors through the map $\begin{bmatrix} f & p \end{bmatrix}$, which shows the claim.
\end{proof}

\begin{proposition} \label{P: fountain cluster tilting}
Let $\mathcal{T}$ in $\cC_2$ be a maximal rigid subcategory containing a fountain. Then $\mathcal{T}$ is functorially finite.
\end{proposition}

\begin{proof}
Let $T$ be the triangulation of the completed $\infty$-gon corresponding to $\mathcal{T}$. We assume without loss of generality that $T$ has a fountain at $0$. Precovering and preenveloping for finite arcs in $\underline{\cC}_2$ was shown in \cite{HJ-cat} and follows in our situation from Lemma \ref{l:lifting covers}, by lifting their precovers and adding the infinite arc. It remains to show that any infinite arc in the completed $\infty$-gon has a precover and a preenvelope in $\mathcal{T}$. Let thus $\gamma = (-\infty,l)$ be an infinite arc. We assume that $l > 0$, the case $l < 0$ follows symmetrically, and the case $l = 0$ is trivial, given that $(-\infty,0) \in T$.\\

By Lemma \ref{l:lifting covers}, it suffices to work in the stable category $\underline{C}_2$. Let $F = \{(0,b) \in T \mid b \geq 0\} \cup \{(-\infty, 0)\}$. By Proposition \ref{prop:exts} and Remark \ref{R:stable Hom} there are only finitely many arcs $\alpha \in T \setminus F$ having stable non-trivial morphisms from or to $\gamma$. Thus, we are done if we can show that there is an $F$-preenvelope and an $F$-precover of $\gamma$. \\
\emph{Precover.} By Lemmas \ref{L: infinite to infinite} and \ref{L: finite to infinite}, all stable morphisms from $F$ to $\gamma$ factor through $(-\infty,0)$, which thus provides a precover. \\
\emph{Preenvelope.} Note that there is no map from $\gamma$ to the unique infinite arc $(-\infty,0)$ in $F$. Thus, we only need to consider finite arcs. Let $F_{\leq l} = \bigoplus \{(0,b) \in T \mid b \leq l\}$ and set $M = \bigoplus_{\alpha \in F_{\leq l}} \alpha$. By \ref{L:infinite to finite} and the explicit description of morphisms after Lemma \ref{L:finite to finite}, all morphisms from $\gamma$ to arcs in $F$ factor through $M$.
\end{proof}


\newcommand{\etalchar}[1]{$^{#1}$}

    \end{document}